\documentclass{amsart}

\usepackage{amsmath,amsfonts,amscd,amssymb,color,graphicx}

\theoremstyle{plain}
\newcounter{thmcount}[subsection]

\newtheorem{theorem}[thmcount]{Theorem}
\newtheorem{corollary}[thmcount]{Corollary}
\newtheorem{lemma}[thmcount]{Lemma}
\newtheorem{proposition}[thmcount]{Proposition}
\newtheorem{conjecture}[thmcount]{Conjecture}

\newtheorem*{conjecture*}{Conjecture}
\theoremstyle{definition}
\newtheorem{remark}[thmcount]{Remark}
\newtheorem{example}[thmcount]{Example}
\newtheorem{definition}[thmcount]{Definition}
\newtheorem{notation}[thmcount]{Notation}

\newtheorem*{acknowledgements}{Acknowledgements}

\def\newmathop#1{\expandafter\gdef\csname #1\endcsname{\mathop{\rm #1}\nolimits}}

\newmathop{ord}
\newmathop{rk}
\newmathop{corank}
\newmathop{Ind}
\newmathop{Res}
\newmathop{Reg}
\newmathop{Gal}
\newmathop{tors}
\newmathop{BSD}
\newmathop{im}
\newmathop{codim}
\newmathop{Hom}
\newmathop{End}
\newmathop{Aut}
\newmathop{Sel}
\newmathop{Ann}
\newmathop{tr}
\newmathop{Re}
\newmathop{Disc}

\def\Q{{\mathbb Q}}
\def\F{{\mathbb F}}
\def\Z{{\mathbb Z}}
\def\R{{\mathbb R}}
\def\C{{\mathbb C}}
\def\N{{\mathbb N}}

\def\triv{{\mathbf 1}}
\def\SG{{\mathcal T}}
\def\B{{\mathfrak B}}
\def\A{{\mathcal A}}
\def\L{{\mathcal L}}
\def\a{{\mathbf a}}
\def\cO{{\mathcal O}}
\def\iso{\simeq}
\def\myperp{{\!\vee}}
\def\different{{\mathfrak D}}

\def\Cl{\mathbf{Cl}}
\def\T{{\mathbf T}}
\def\TFKp{{\mathbf T_p^{\scriptscriptstyle F/K}}}
\def\RC{{\mathcal C}}
\def\X{{\mathcal X}}

\def\kI{\rm I}

\def\vI{\hbox{$1\!\!\cdot\!\! 1 \!\!:\>$}}
\def\vII{\hbox{$2\!\!\cdot\!\! 2\!\!:\>$}}
\def\vIIa{\hbox{$1\!\!\cdot\!\! 2A \!\!:\>$}}
\def\vIIb{\hbox{$1\!\!\cdot\!\! 2B \!\!:\>$}}
\def\vIII{\hbox{$3\!\!\cdot\!\! 3\!\!:\>$}}
\def\vIV{\hbox{$4\!\!\cdot\!\! 4\!\!:\>$}}
\def\vVI{\hbox{$6\!\!\cdot\!\! 6\!\!:\>$}}
\def\voI{1 \!:\!}
\def\voII{2 \!:\!}

\def\vIs{\hbox{$1\!\!\cdot\!\! 1$}}
\def\vIIs{\hbox{$2\!\!\cdot\!\! 2$}}

\def\vIIbs{\hbox{$1\!\!\cdot\!\! 2B$}}
\def\vIIIs{\hbox{$3\!\!\cdot\!\! 3$}}
\def\vIVs{\hbox{$4\!\!\cdot\!\! 4$}}
\def\vVIs{\hbox{$6\!\!\cdot\!\! 6$}}

\def\neron{\omega^{\circ}}

\def\Lak{{\Lambda_{\scriptstyle A/K}}}
\def\Lakv{{\Lambda^\vee_{\scriptstyle A/K}}}
\def\Lakf{{\Lambda_{\scriptstyle A/K}^F}}
\def\Lakvf{{\Lambda^{\vee F}_{\scriptstyle A/K}}}

\def\Lal{{\Lambda_{\scriptstyle A/L}}}
\def\Lalv{{\Lambda^\vee_{\scriptstyle A/L}}}
\def\Lalf{{\Lambda_{\scriptstyle A/L}^F}}
\def\Lalvf{{\Lambda^{\vee F}_{\scriptstyle A/L}}}

\def\Lek{{\Lambda_{\scriptstyle E/K}}}
\def\Lekv{{\Lambda^\vee_{\scriptstyle E/K}}}

\renewcommand{\div}{\ensuremath{\:|\:}}

\def\leftchoice#1#2#3#4{{\def\arraystretch{0.7}
\Bigl\{\!\!\begin{array}{ll}
   \scriptstyle #1,\!\!\!&\scriptstyle #2\cr
   \scriptstyle #3,\!\!\!&\scriptstyle #4\end{array}}}

\def\leftchoicethree#1#2#3#4#5#6{{\def\arraystretch{0.7}
\Biggl\{\!\!\begin{array}{ll}
   \scriptstyle #1,\!\!\!&\scriptstyle #2\cr
   \scriptstyle #3,\!\!\!&\scriptstyle #4\cr
   \scriptstyle #5,\!\!\!&\scriptstyle #6\end{array}}}

\def\smallmatrix#1#2#3#4{
  \genfrac{(}{.}{0pt}{1}{#1}{#3}
  \genfrac{.}{)}{0pt}{1}{#2}{#4}
}

\input cyracc.def       
\font\tencyr=wncyr10
\def\sha{\text{\tencyr\cyracc{Sh}}}
\font\eightcyr=wncyr8

\def\smallsha{\text{\eightcyr\cyracc{Sh}}}

\def\UDspace{\vphantom {\int^{\int^\int}}}
\def\rh#1{\smash{\raise-5pt\hbox{$#1$}}}
\def\rhv#1#2{\smash{\raise#1pt\hbox{#2}}}

\begin{document}

\title[{$\Z[C_n]$-lattices and Tamagawa numbers}]{Finite quotients of $\Z[C_n]$-lattices and Tamagawa numbers of semistable abelian varieties}
\author{L. Alexander Betts, Vladimir Dokchitser}
\address{Merton College, Oxford OX1 4JD, United Kingdom}
\email{alexander.betts@maths.ox.ac.uk}
\address{Mathematics Institute, University of Warwick, Coventry CV4 7AL, United Kingdom}
\email{v.dokchitser@warwick.ac.uk}
\address{Mathematics Institute, University of Warwick, Coventry CV4 7AL, United Kingdom}
\email{a.morgan.1@warwick.ac.uk}


\maketitle

\begin{center}
(with an appendix by V. Dokchitser and A. Morgan)
\end{center}

\begin{abstract}
We investigate the behaviour of Tamagawa numbers of semistable principally
polarised abelian varieties in extensions of local fields. In view of
the Raynaud parametrisation, this translates into a purely algebraic
problem concerning the number of $H$-invariant points on a quotient
of $C_n$-lattices $\Lambda/e\Lambda'$ for varying $H\le C_n$ and $e\in\N$.
In particular, we give a simple formula for the change of Tamagawa numbers
in totally ramified extensions (corresponding to varying $e$) and one
that computes Tamagawa numbers up to rational squares in general extensions.

As an application, we
extend some of the existing results on the $p$-parity conjecture for Selmer
groups of abelian varieties by allowing more general local behaviour.
We also give a complete classification of the behaviour of Tamagawa numbers
for semistable 2-dimensional principally polarised abelian varieties, that
is similar to the well-known one for elliptic curves.
The appendix explains how to use this classification for Jacobians of
genus 2 hyperelliptic curves given by equations of the form $y^2=f(x)$,
under some simplifying hypotheses.
\end{abstract}

\tableofcontents



\section{Introduction}

The present article has a number theory side and an algebraic side. 
From the number theory point of view, we study local Tamagawa numbers of abelian varieties over $p$-adic fields, and how these change as one makes the field larger.
We will focus on the relatively simple case of semistable abelian varieties. 
However, it is crucial to us that the residue field of the base field is not algebraically closed,
in which respect our focus differs from many works, such as \cite{Lor,NU} or more recently \cite{HN}.
Indeed, the origin of our investigation came from the study of Selmer groups and questions related to the Birch--Swinnerton-Dyer conjecture, the parity conjecture and Iwasawa theory, where the abelian varieties live over number fields and the local terms, such as Tamagawa numbers, are computed over their completions.
One of our main applications is an extension of the results of \cite{tamroot} on the $p$-parity conjecture.

The theory of semistable abelian varietes allows us to translate this study of their Tamagawa numbers into one of pure algebra.
Thus, from the algebraic point of view, we study finite quotients $\Lambda/\Lambda'$ of an integral representation of a cyclic group $C_n$ by a finite index submodule. Specifically, we develop tools for studying the $C_n$-invariant elements of this quotient, and how their number changes as one refines this quotient module to $\Lambda/e\Lambda'$ for varying $e\in\Z$.

\subsection{Algebraic setting}
\label{ss:introalg}

The reader who is only interested in applications to abelian varieties may wish to skip to \S\ref{ss:introtam} at the expense of not seeing the formal definition of the group $\B$, and instead treating it as a black box.

Our algebraic setting is the following: $\Lambda\iso \Z^{\oplus d}$ will be a finitely generated free abelian group and $\Lambda'\subseteq\Lambda$ a subgroup of finite index; we will refer to them as {\em lattices} even when we don't insist that they come with a bilinear form.
We will primarily be interested in the case when $F$ is an automorphism of $\Lambda$ of finite order that preserves $\Lambda'$. However, for the purposes of the main theory this is unnecessarily restrictive: we only need $F$ to be a semisimple endomorphism (that maps $\Lambda'$ to itself), or even just an endomorphism whose eigenvalue-1 generalised eigenspace is not larger than its 1-eigenspace, i.e. its Jordan blocks for eigenvalue 1 are all of size 1, or, equivalently, $\ker(F\!-\!1)=\ker(F\!-\!1)^2$.
For the main applications, $\Lambda$ will be the dual lattice to $\Lambda'$ with respect to an $F$-invariant symmetric non-degenerate $\Z$-valued pairing on $\Lambda'$.

We would like to understand the $F$-invariant points in the quotient group $\Lambda/\Lambda'$, and how these change as we refine the quotient by scaling $\Lambda'$ to $e\Lambda'$ for $e\in\Z$.
To see the kinds of behaviour that can occur consider the following examples for~$\Lambda=\Z^{\oplus 2}$. (i) First take $\Lambda=\Lambda'$ and $F$ the identity map. Then $\Lambda/e\Lambda'$ is, of course, all $F$-invariant and its order grows as $e^2=e^{\rk\Lambda^F}$; in general the sublattice of $F$-invariants $\Lambda^F$ will always contribute to $(\Lambda/e\Lambda')^F$ with precisely this kind of growth as $e$ varies. (ii) Next consider the ``diagonal'' index 2 sublattice of $\Lambda$, that is take $\Lambda'$ to be spanned by the vectors $(1,1)$ and $(1,-1)$. In this case both points of $\Lambda/\Lambda'$ must be $F$-invariant, even if $F$ has no invariants on $\Lambda$ itself. Note that these survive as $F$-invariants in $\Lambda/e\Lambda'\cong\frac{1}{e}\Lambda/\Lambda'$, and hence give a constant contribution as $e$ varies. (iii) Finally, consider $\Lambda'=\Lambda$ and $F$ the multiplication by -1 map. When $e$ is odd, $(\Lambda/\Lambda')^F$ is trivial, but when $e$ is even $(0,0), (\frac e2,0), (0,\frac e2)$ and $(\frac e2,\frac e2)$ are elements of order 2 in $\Lambda'/e\Lambda'$ and are, perforce, $F$-invariant. In other words, $\Lambda/e\Lambda'$ can pick up a few sporadic $F$-invariants when $e$ becomes divisible by a certain integer.

It turns out that these are the only types of contributions to $(\Lambda/e\Lambda')^F$.
For this purpose we will introduce two groups: 
$$
\SG =\SG(\Lambda)=\frac\Lambda{\Lambda^F+(\Lambda\cap V^{F,\perp})}
 \qquad {\text{and}} \qquad 
\B=\B_{\Lambda,\Lambda'}=\frac{(V/\Lambda')^F \mod \Lambda}{V^F \mod \Lambda},
$$
where
$V=\Lambda\otimes_\Z\Q$ regarded as a vector space with an $F$-action and containing $\Lambda$, and  $V^{F,\perp}$ is the orthogonal complement to $V^F$ under a non-degenerate $F$-invariant pairing (or the sum of the generalised eigenspaces of $F$ on $V$ for all eigenvalues except 1, in the absence of such a pairing).
The ``separation group'' $\SG$ simply measures how far $\Lambda$ is away from being the direct sum of its $F$-invariant sublattice and its orthogonal complement (or suitable equivalent in the absence of an $F$-invariant pairing).
The group $\B$ is more subtle and is responsible for the sporadic growth of $F$-invariants in example (iii) above. Indeed, if $\Lambda=\Lambda'$ and $F$ has no invariants on $\Lambda$ (so $V^F=0$), then $\B=(V/\Lambda)^F\! \mod \Lambda$: this precisely recovers the group $C_2\times C_2$ of 2-torsion points in example (iii).

The general result that we will prove is the following; here $\SG'=\SG(\Lambda')$ refers to the separation group of $\Lambda'$:

\begin{theorem}\label{thm:exactformulafixpoint}
Suppose that $\Lambda$ is a lattice with an endomorphism $F$ satisfying $\ker(F\!-\!1)=\ker(F\!-\!1)^2$.
Suppose $\Lambda'\subseteq\Lambda$ is a full-rank $F$-stable sublattice.
Write the characteristic polynomial of $F$ over $\Lambda$ as $\pm(t-1)^rp(t)$ where $p(1)>0$. 
Then
$$
\left|\left(\!\frac\Lambda{\Lambda'}\!\right)^{\!\!\!F}\right|=
\left|\frac{\Lambda^F}{\Lambda'^F}\right|   \cdot   \frac{p(1)}{|\B|\cdot|\SG'|},
$$
and for all $e\ge 1$,
$$
\left|\left(\!\frac\Lambda{e\Lambda'}\!\right)^{\!\!F}\right|=
\left|\left(\!\frac\Lambda{\Lambda'}\!\right)^{\!\!F}\right|\cdot
|\B[ e]|\cdot  e^r.
$$
$\SG'$ and $\B$ are finite abelian groups, the product of whose orders divides $p(1)$. 
If $\Lambda'$ is the dual lattice to $\Lambda$ with respect to a non-degenerate $F$-invariant symmetric pairing, then 
$\SG\iso\SG'$ and $\B$ admits a perfect
antisymmetric pairing; in particular the order of $\B$ and of $\B[e]$ is either a square or twice a square.
\end{theorem}

The first formula in the theorem above is proved in Theorem \ref{thm:exactformulatorsionnoe}, 
the second one in Corollary \ref{cor:exactformulatorsion}, and the result on the finiteness and 
orders of $\SG'$ and $\B$ in Corollary \ref{cor:BTdivP} (with $D=F\!-\!1$ and $V=\Lambda\otimes_\Z\Q$).
The results on $\SG$ and $\B$ when $\Lambda'$ is the dual lattice to $\Lambda$ are proved in
Theorems \ref{lem:perppreservessepgps} and \ref{thm:bettspairing} (note that $F$ is necessarily an automorphism of $V=\Lambda\otimes\Q$ if it preserves a non-degenerate pairing).
Further properties of the groups $\SG$ and $\B$, including results on their exponents and number of generators, are proved in \S\ref{ss:sepgp}--\ref{ss:pairing}.

For applications to Tamagawa numbers, $F$ will be an automorphism of finite order and $\Lambda'$ the dual to $\Lambda$ with respect to a symmetric pairing, and we will both be interested in varying $e$ and in replacing $F$ by $F^f$, i.e. decreasing the order of the group of automorphisms that it generates.
Theorem \ref{thm:uptosquares2} shows that, up to a square error, $|(\Lambda/e\Lambda')^{F^f}|$ changes in a very elementary fashion: it essentially only depends on the parity of $e$ and $f$.
In \S \ref{ss:c_n} we also give a classification of all possible such 2-dimensional $F$-lattices together with their invariants $\SG$ and $\B$, and all lattices where $\Lambda\otimes\Q$ is the (unique) faithful irreducible $\Q[F]$-representation (Theorems \ref{thm:2d} and \ref{thm:c_n}).

\subsection{Tamagawa numbers}
\label{ss:introtam}

The reason for developing the above algebraic machinery is that we wish to apply it to study the behaviour of Tamagawa numbers of abelian varieties in field extensions. Suppose $A/K$ is a principally polarised abelian variety over a $p$-adic field with semistable reduction. If the reduction is split semistable, then its Tamagawa number 
can be expressed as $c_{A/K}=|\Lak/\Lakv|$, where $\Lakv$ and $\Lak$ are the character group of the toric part in the Raynaud parametrisation of $A$ and its dual lattice with respect to the monodromy pairing (see \S\ref{s:tam} or \cite{tamroot} \S3.5.1 for further details). If the reduction is semistable, but not necessarily split, then the Frobenius element $F\in\Gal(K^{nr}/K)$
acts as an automorphism on $\Lak$ and $\Lakv$, and the Tamagawa number is
$$
  c_{A/K} = \left|\biggl(\frac{\Lak}{\Lakv}\biggr)^{\!\!F}\right|.
$$ 
Taking the base change of $A$ to an extension with residue degree $f$ and ramification degree $e$ effectively replaces $F$ by $F^f$ and $\Lakv$ by $e\Lakv$ in the formula for the Tamagawa number. This is precisely the setting in which our algebraic results can be applied.

Thus the main invariant of $A/K$ in our description of the behaviour of Tamagawa numbers will be the group 
$$
  \B_{A/K}=\B_{\Lak,\Lakv}.
$$
Recall that this is a finite abelian group which, in our setting, admits a perfect antisymmetric pairing (so, in particular, has square or 2$\times$square order).

The theorem below illustrates the type of results that we obtain; the full list is given in \S \ref{ss:tam}. Note that point (i) interprets $\B$ number-theoretically, as the group that controls the Tamagawa number $c_{A/L}$ in totally ramified extensions $L/K$.
Points (ii) and (iii) were motivated by, respectively, Iwasawa theoretic considerations (see \cite{megasha} Thm.\ 5.5) and applications to the $p$-parity conjecture that require knowing the behaviour of Birch--Swinnerton-Dyer quotients up to rational squares in field extensions (see \S\ref{ss:intro-p} below, and \S\ref{sspparity}).

Let us say that $A/K$ has {\em toric dimension} $d$ and {\em split toric dimension} $r$ if $\rk \Lak\!=\!d$ and $\rk\Lakf\!=\!r$. Equivalently, $r$ is the multiplicity of $1\!-\!T$ in the local polynomial of $A/K$ (essentially the Euler factor in the $L$-series) or, alternatively, the multiplicity of $1$ in the characteristic polynomial of the Frobenius element in its action on the Tate module of $A/K$ (or on $H^1_{{\text{\'et}}}(A,\Q_l)$). Similarly $d$ is the total number, counting multiplicities, of such roots of absolute value 1.

\begin{theorem}\label{thm:cvintro}
Let $K$ be a finite extension of $\Q_p$, and $A/K$ a semistable principally polarised abelian variety of toric dimension $d$ and split toric dimension $r$.

(i) If $L/K$ is a totally ramified extension of degree $e$, then
$$
  c_{A/L} = |\B_{A/K}[e]| \cdot c_{A/K} \cdot e^{r}.
$$

(ii) If $K\subset L_1 \subset L_2 \subset \ldots$ is a tower of finite extensions with $L_k/K$ of ramification degree $e_k$, then for all sufficiently large $k$
$$
  c_{A/L_k} = C \cdot e_k^{r_\infty}, 
$$
for some suitable constant $C\in\Q$, and where $r_{\infty}$ is the split toric dimension of $A/L_k$ for all sufficiently large $k$.

(iii) If $L/K$ is a finite extension of residue degree $f$ and ramification degree $e$, then
$$
 c_{A/L}  \sim
    \begin{cases}
      c_{A/K} \cdot e^r
        & \text{if }\>2\nmid e,2\nmid f, \\
      c_{A/K}\cdot|\B_{A/K}|\cdot e^r
        & \text{if }\>2\div e,2\nmid f, \\
      c_{A/K^{nr}}\cdot e^d
        & \text{if }\>2\div f,
    \end{cases}
$$
where $\sim$ denotes equality up to rational squares.
\end{theorem}

The above result follows from Corollary \ref{cor:Bram}, Corollary \ref{cor:iwastab} and Theorem~\ref{thm:cvuptosquares}, which
in turn rely on the properties of $\B$ and related invariants of $\Z[C_n]$-lattices established in \S\ref{s:betts}.
It is worth remarking that for the purposes of point (ii) the assumption that $A/K$ admits a principal polarisation may be removed
(see Corollary~\ref{cor:iwastab}).
However, the semistability assumption is necessary to obtain such a stable growth: for example, the Tamagawa number of the elliptic curve 243a1 fluctuates between 1 and 3 in the layers of the $\Z_3$-cyclotomic tower of $\Q_3$, see \cite{megasha} Remark~5.4.

We also apply our algebraic machinery to study in detail semistable abelian varieties that have toric dimension 2 (see \S\ref{ss:dim2}), and to those for which $\Lak\otimes\Q$ is an irreducible $\Z[F]$-module (equivalently those for which the eigenvalues of the Frobenius element $F$ on $\Lak$ are the set of primitive $k$-th roots of unity for some~$k$); see Remark \ref{rmk:cncv}. In particular we will give a complete classification of reduction types and their behaviour in field extension for the case of toric dimension 2: this is summarised in the theorem below, which is a direct consequence of Theorems~\ref{thm:2d} and \ref{thm:cvdim2}.
It is worth pointing out that this is an altogether different kind of classification from that of Namikawa--Ueno \cite{NU} for curves of genus 2, since we treat only the semistable case, but, on the other hand, are very much interested in the action of the Frobenius element.
To motivate the result, let us first recall the corresponding classification for elliptic curves:

\begin{example}
Let $K/\Q_p$ be a finite extension and $E/K$ an elliptic curve with multiplicative reduction. The elliptic curve has Kodaira type $\kI_n$ for some $n\ge 1$, and either split or non-split multiplicative reduction. If $L/K$ is a finite extension of residue degree $f$ and ramification degree $e$, then $E/L$ has Kodaira type $\kI_{en}$ and the split/non-split characteristic is the same as for $E/K$ unless $f$ is even, in which case the reduction is always split. The reduction type is related to the $\Z[F]$-structure of $\Lek$ and $\Lekv$ as follows:

\begingroup\smaller[1]
$$
\begin{array}{|l|c|c|c|c|c|c|  }
 \hline
 \text{Type}^{\phantom {X^X}} & \Lek         & \Lekv            & F & c_{E/K} &  f=2  & \text{condition}\cr
 \hline
 \kI_n \text{ split}   & \Z      & n\Z      & 1 & n &\text{unchanged}&   \cr
 \hline
 \rh{\kI_n \text{ non-split}} & \rh\Z      & \rh{n\Z}     & \rh{-1} &  1 & \rh{\kI_n \text{ split}} & n \text{ odd}     \cr
           &                 &                      &    &   2 && n \text{ even}     \cr
\hline
\end{array}\\[7pt]
$$
\endgroup
The table also lists the corresponding Tamagawa number and the effect of a quadratic unramified extension (under ``$f=2$'').
Note that together with the fact that an unramified extension of odd degree does not change the reduction type and the knowledge of the effect of a totally ramified extension, the table is sufficient to determine the reduction type of $E/L$ and its Tamagawa number for every finite extension $L/K$.

Note also that the above classification applies more generally to semistable abelian varieties of toric dimension 1: $\Lak$ is then isomorphic to $\Z$, $\Lakv$ is some sublattice (hence $=n\Z$) and $F$ is an automorphism of $\Lak$ (hence multiplication by $\pm 1$). As explained in the beginning of the section, the Tamagawa number is given by $|(\Z/n\Z)^F|$, and the behaviour in ramified and unramified extensions is obtained by scaling $e$ or replacing $F$ by $F^f$. Thus the overall result is exactly as for elliptic curves with multiplicative reduction. 

The theorem below gives an analogous table for toric dimension 2. In particular, the two results together cover Jacobians of all semistable curves of genus 2. The appendix (\S\ref{s:appendix}) explains how to explicitly determine the type that occurs for Jacobians of genus 2 hyperelliptic curves given by equations of the form $y^2=f(x)$ over finite extensions of $\Q_p$, assuming that $p$ is odd and that $f(x)$ is integral with a unit leading coefficient
and no triple roots over the residue field.

\end{example}

\begin{theorem}\label{thm:introclassification}
Let $K/\Q_p$ be a finite extension and let $A/K$ be a semistable principally polarised abelian variety of toric dimension 2. 

Then, up to $\Z[F]$-isomorphism, $\Lak,\Lakv$ and the action of $F$ on $\Lak$ are given by one of the cases in the following table.
The parameters $n$ and $m$ are strictly positive integers.
All the types are distinct, except for $\vI n,m$ and $\vI n',m'$ which are isomorphic if and
only if $C_n\times C_m \iso C_{n'}\times C_{m'}$, and similarly for $\vII n,m$ and $\vII n',m'$.
The Tamagawa number $c_{A/K}$ is determined by the type as shown in the table.
\begingroup\smaller[1]
$$
\hskip-11mm
\begin{array}{|l|c|c|c|c|c|c|c|  }
 \hline
\UDspace
 \text{Type} & \Lak         & \Lakv            & F & c_{A/K} &  f=2 & f=3  & \text{condition}  
\\[2pt] \hline \UDspace 
  \vI n,m  & \Z^2      & n\Z\oplus m\Z        &\smallmatrix1001& nm &\text{unchanged}&\text{unchanged}&   
 \\[2pt] \hline \UDspace 
  \rh{\vIIa n,m} & \rh{\Z^2}      & \rh{n\Z \oplus m\Z}       & \rh{\smallmatrix100{-1}}&  n & \rh{\vI n,m} & \rh{\text{unchanged}}& m \text{ odd}     \cr
           &                 &                      &    &   2n&&& m \text{ even}     
\\[2pt] \hline \UDspace 
   &  & &&&
      \vI 2n,\frac{m}{2} & & \ord_2\frac nm>0      \cr
  \vIIb n,m & \Z^2+\langle \frac12,\frac12\rangle & n\Z\oplus m\Z+\langle \frac n2,\frac m2\rangle  &\smallmatrix100{-1}& n &
      \vI n,m &\text{unchanged}& \ord_2\frac nm= 0      \cr
   \rhv{3}{${\scriptstyle n\equiv m \!\!\!\mod 2}$}&  & &&&
      \vI \frac n2,2m & & \ord_2\frac nm< 0     
\\[2pt] \hline \UDspace 
           &      &                 &                              &   1&&& n, m \text{ odd}    \cr
 \vII n,m  & \Z^2  & n\Z \oplus m\Z      &\smallmatrix{-1}00{-1}& 2&\vI n,m&\text{unchanged}& n,m \text{ odd/even}     \cr
           &      &                 &                              &   4&&& n,m \text{ even}    
\\[2pt] \hline \UDspace 
 \vIII n   & \Z[\zeta_3] & (\zeta_3-1)n\Z[\zeta_3] &\cdot\zeta_3 &3&  \text{unchanged}& \vI n,3n & 
\\[2pt] \hline \UDspace 
 \rh{\vIV n}    & \rh{\Z[i]} & \rh{n\Z[i]} &\rh{\cdot i}&1  & \rh{\vII n,n}&\rh{\text{unchanged}}& n \text{ odd}   \cr
           &                    &       &        &  2& & & n \text{ even}  
\\[2pt] \hline \UDspace 
 \vVI n    & \Z[\zeta_3] & (\zeta_3-1)n\Z[\zeta_3] &\cdot\zeta_6& 1& \vIII n & \vII n,3n  & 
\\[2pt] \hline
\end{array}\\[5pt]
$$
\endgroup

If $L/K$ is an unramified extension of degree coprime to 2 and 3 then $A/L$ has the same type with the same parameters $n, m$ as $A/K$. If $L/K$ is an unramified extension of degree 2 or 3 then the type of $A/L$ changes as shown in the table under the headings ``$f=2$'' and ``$f=3$'', respectively. Finally, if $L/K$ is a totally ramified extension of degree $e$, then $A/L$ has the same type as $A/K$ with parameters $n$ and $m$ replaced by $ne$ and $me$. 

\end{theorem}

\subsection{On the $p$-parity conjecture}
\label{ss:intro-p}

Finally, we turn to an application to the $p$-parity conjecture for abelian varieties, which is a parity version of the Birch--Swinnerton-Dyer conjecture.

Let $A/K$ be an abelian variety over a number field. Its main arithmetic invariant is its Mordell--Weil rank, that is the rank of its group of $K$-rational points $A(K)$. The Birch--Swinnerton-Dyer conjecture predicts that this coincides with the analytic rank of $A/K$ --- the order of vanishing at $s=1$ of the $L$-function $L(A/K,s)$. Since the $L$-function is conjectured to satisfy a functional equation of the type $s \leftrightarrow 2-s$, the parity of the analytic rank can be read off from the sign in the functional equation: it is even if the sign is $+$ and odd if the sign is $-$. This sign is (conjecturally) explicitly given as the root number $w(A/K)$, which is constucted via the theory of local root numbers or local $\epsilon$-factors, and hence one can talk about the ``parity of the analytic rank'' without knowing that the $L$-function has an analytic continuation to $s=1$. The Parity Conjecture is precisely this parity version of the Birch--Swinnerton-Dyer Conjecture, i.e. that the parity of the rank of $A/K$ can be read off from the root number as $(-1)^{\rk A/K}=w(A/K)$.

Moreover, if $F/K$ is a Galois extension, then one can try to  determine the multiplicities of the irreducible constituents in the representation $A(F)\otimes_\Z\C$ of $\Gal(F/K)$. A generalisation of the Birch--Swinnerton-Dyer conjecture predicts that the multiplicity of a (complex) irreducible representation $\tau$ is given by the order of vanishing of the twisted $L$-function $L(A/K,\tau,s)$ at $s=1$, see e.g. \cite{RohV} \S2. Once again, there is a parity conjecture to the effect that, for self-dual $\tau$, the parity of this multiplicity can be recovered from the associated root number as $(-1)^{\langle\tau,A(K)\otimes\C\rangle}=w(A/K,\tau)$, where $\langle,\rangle$ is the usual inner product of characters.

In view of the conjectured finiteness of the Tate-Shafarevich group $\sha(A/K)$, there is also the following $p$-parity conjecture, that morally ought to be more accessible, although still remains unresolved. For a prime number $p$ let
$$
  \X_p(A/K) = (\text{Pontryagin dual of the } p^{\infty}\text{-Selmer group of } A/K)\otimes_{\Z_p}\Q_p.
$$
This is a $\Q_p$-vector space whose dimension is the rank of $A/K$, provided $\sha(A/K)$ is finite; in general $\dim_{\Q_p}\X_p(A/K)=\rk A/K + t$, where $t$ measures the $p$-divisible part of $\sha$ as $\sha(A/K)[p^{\infty}]=(\Q_p/\Z_p)^{\oplus t}\oplus \text{(finite group)}$.
Then we expect:

\begin{conjecture}[$p$-Parity Conjecture]
\label{conj:pparity}
For an abelian variety over a number field $A/K$,
$$
  (-1)^{\dim \X_p(A/K)}=w(A/K).
$$
If $F/K$ is a Galois extension and $\tau$ a self-dual representation of $\Gal(F/K)$, then
$$
 (-1)^{\langle \tau, \X_p(A/F) \rangle} = w(A/K,\tau).
$$
\end{conjecture}

\noindent The second formula is a strictly stronger statement, as taking $\tau=\triv$ it recovers the first one.
We will refer to it as the $p$-parity conjecture for the twist of $A/K$ by $\tau$.

Most results on the $p$-parity conjecture concern elliptic curves; in particular the first formula is now known to hold for all elliptic curves over $\Q$ and in many cases for elliptic curves over totally real fields \cite{kurast, NekIV}. 
The situation with general abelian varieties is more difficult, and the main results are those of \cite{CFKS} that establish the first formula assuming that the abelian variety admits a suitable isogeny, and of \cite{tamroot} that proves the second formula for a class of representations $\tau$. Our results on the behaviour of Tamagawa numbers let us strengthen the results of the latter paper as follows.

For a Galois extension of number fields $F/K$ and a prime number $p$, the ``regulator constant'' machinery of \cite{tamroot} and its preceding papers constructs a special set $\TFKp$ of self-dual representations of $\Gal(F/K)$. Unfortunately there is still no simple description of the set $\TFKp$; we refer the reader to \S\ref{sspparity} for its definition.
We will prove the following result on on $p$-parity conjecture for twists by $\tau\in\TFKp$. It effectively removes 
the ugly assumption ``4ex'' from \cite{tamroot} Thm.\ 3.2 and its applications.
Strictly speaking this is only a strengthening of existing results for the one prime $p\!=\!2$. 
However, $2$-Selmer groups have been particularly important for the parity conjectures, explicit descent, quadratic twists, as well as recent results on hyperelliptic curves \cite{Manjul-Gross-Wang}.

\begin{theorem}
\label{thm-introparity}
Let $F/K$ be a Galois extension of number fields and let $p$ be a prime
number. Let $A/K$ be a principally polarised abelian variety
all of whose primes of unstable reduction have cyclic
decomposition groups\footnote{e.g. if they are unramified}
in $F/K$; if $p=2$ assume also that the
polarisation is induced by a $K$-rational divisor.

\begin{enumerate}
\item The $p$-parity conjecture holds for all twists of $A/K$
by $\tau\in \T^{\scriptscriptstyle{F/K}}_p$. 

\item If $\Gal(F/K) \iso D_{2p^n}$, then the $p$-parity
conjecture holds for $A/K$ twisted by $\tau$ of the form
$\tau=\sigma\oplus\triv\oplus\det\sigma$ for every 2-dimensional
representation $\sigma$ of $\Gal(F/K)$.
(Here $D_{2k}$ denotes the dihedral group of order $2k$.)

\item
Suppose $p=2$ and that the $2$-Sylow subgroup of $\Gal(F/K)$ is normal.
If the $2$-parity conjecture holds for $A/K$ and over all quadratic
extensions of $K$ in $F$, then it holds for $A$ over all subfields of $F$
and for all twists of $A$ by orthogonal representations of $\Gal(F/K)$.
\end{enumerate}
\end{theorem}

Let us stress that the theorem provides a large number of twists of $A/K$ for which the $p$-parity conjecture holds: for example we could prove the $p$-parity conjecture for all principally polarised abelian varieties $A$ over $\Q$ using (2) if we knew how to find a $D_{2p}$-extension of $\Q$ in which (i) the $p^\infty$-Selmer rank of $A$ does not grow, (ii) the analytic rank of $A$ does not grow, or at least $w(A,\sigma)=w(A,\det\sigma)=1$ for a faithful 2-dimensional representation $\sigma$, and (iii) the primes of unstable reduction of $A$ are unramified. (Here $D_4$ should be interpreted as $C_2\times C_2$ for $p=2$.)

Part (1) will be proved in Theorem \ref{thm-parity}. Part (2) then follows from \cite{tamroot} Ex.\ 2.21, 2.22 that compute regulator constants for dihedral groups (same as in \cite{tamroot} Thm.~4.2). Similarly (3) follows from \cite{tamroot} Thm.\ 4.4, Thm.\ 4.5., where ``Hypothesis~4.1'' holds by~(2).

\begin{remark}
Adam Morgan has recently proved that the $2$-parity conjecture holds after a proper quadratic extension of the base field for Jacobians of hyperelliptic curves, under some local conditions on the reduction types \cite{Mor}. Thus in Theorem \ref{thm-introparity} (3) for this class of abelian varieties one only needs to assume the 2-parity conjecture for $A/K$ itself.
\end{remark}


\subsection{Notation}

Throughout this paper, a {\em lattice} $\Lambda$ will simply mean a finitely generated free abelian group, $\Lambda\iso\Z^n$. It will often come with a non-degenerate symmetric pairing $(\cdot,\cdot):\Lambda\times\Lambda\to\Q$, but in such cases we will always explicitly state it.
A $\Z[F]$- or $\Z[G]$-lattice will mean a lattice with an endomorphism $F$ or with a linear action of a finite group $G$, respectively.

\begin{definition}\label{def:dual}
If $V$ is a finite-dimensional $\Q$-vector space endowed with a non-degenerate symmetric bilinear form
$(\cdot,\cdot):V\times V\rightarrow\Q$, then for any $S\subseteq V$ we
write
$$
 S^{\myperp}=\{v\in V|\forall s\in S:(v,s)\in \Z \}.
$$
Note that if $U\subseteq V$ is a subspace, then $U^{\myperp}$ is just the subspace orthogonal to~$U$.
If $\Lambda\subset V$ is a lattice of maximal rank, then $\Lambda^{\myperp}$ is precisely its dual lattice.

\end{definition}


\noindent Throughout the paper, we use the notation:

\begin{tabular}{llll}
\vphantom{$\int^X$}%
$M^F$ & $F$-invariants of $M$, i.e. $\{m\in M| Fm=m\}$ for an endomorphism $F$ \cr
$M^G$ & $G$-invariants of $M$, i.e. $\{m\in M| gm=m\> \forall g\in G\}$ for a group $G$ \cr
$M[D]$ & $D$-torsion of $M$, i.e. $\{m\in M| Dm=0\}$  for an endomorphism $D$ \cr
$M[e]$ & $e$-torsion of $M$, i.e. $\{m\in M| em=0\}$  for $e\in\Z$ \cr
$\zeta_k$ & primitive $k$-th root of unity \cr
$\Phi_k$ & $k$-th cyclotomic polynomial \cr
$\SG$ & see Definition \ref{def:SG}\cr
$\B$ & see Definition \ref{defn:bettsgp}
\end{tabular}

\noindent For a finite extension $K/\Q_p$ and an abelian variety $A/K$ we use the notation:

\begin{tabular}{llll}
\vphantom{$\int^X$}%
$K^{nr}$ & maximal unramified extension of $K$ \cr
$F$ & Frobenius element (in \S\ref{s:tam}) \cr
$c_{A/K}$ & local Tamagawa number of $A/K$ \cr 
$w(A/K)$ & local root number of $A/K$ \cr
$w(A/K,\tau)$ & local root number of the twist of $A$ by $\tau$, see \cite{RohG} \cr
$\Lak$ & see Notation \ref{not:lak}\cr
$\B_{A/K},  \SG_{A/K}, P_{A/K}$ & see Notation \ref{not:bak}\cr
\end{tabular}

\smallskip
\noindent Note that we use the same notation for global root numbers when $K$ is a number field.
In this setting we also write $\X_p(A/K)$ for the dual $p^\infty$-Selmer group of $A/K$ (see \S\ref{ss:intro-p}),
that is
$$
 \X_p(A/K)=\Hom_{\Z_p}(\varinjlim \Sel_{p^n}(A/K), \Q_p/\Z_p)\otimes\Q_p.
$$ 
Finally, when working with number fields, we will typically write $F$ for a Galois extension of $K$. Hopefully this clash of notation with the Frobenius element in the case of local fields (or its analogue $F\in\Aut(\Lambda)$ in the algebraic setting) will not lead to any confusion.

\bigskip

\begin{acknowledgements}
We would like to thank Tim Dokchitser and Adam Morgan for their useful comments.
We would also like to thank Emmanuel College and Trinity College, Cambridge, where most of this research was carried out.
The second author is supported by a Royal Society University Research Fellowship.
\end{acknowledgements}


\section{Finite quotients of $\Z[F]$-lattices}
\label{s:betts}

In this section we develop the algebraic tools for studying the $F$-invariants of quotients $\Lambda/e\Lambda'$ of $\Z[F]$-lattices. The setting will be slightly more general than in \S\ref{ss:introalg} in the introduction.
We will typically have $V$ a finite-dimensional $\Q$-vector space, $\Lambda\subset V$ a lattice of maximal rank, and $D\in\End(V)$ an endomorphism that preserves $\Lambda$ and satisfies $V[D]=V[D^2]$. The endomorphism $D$ plays the role of $F\!-\!1$ for $F$ as in the introduction and in the applications. Note that, in particular, $V[D]=V^F$ and $\Lambda[D]=\Lambda^F$, and that if $F$ has finite order then both $F$ and $D$ act semisimply on $V$, so that the assumption that $V[D]=V[D^2]$ is automatically satisfied.


\subsection{The separation group $\SG(\Lambda)$}\label{ss:sepgp}

In this section we introduce the group $\SG(\Lambda)$ and establish its basic properties.

\begin{definition}[Separation group $\SG(\Lambda)$]\label{def:SG}
Suppose $V=U\oplus W$ is a finite-dimensional $\Q$-vector space and
$\pi_U,\pi_W$ the projections inducing this direct sum.
For a lattice $\Lambda\subset V$ of full rank we define the
{\em separation group}
$$
\SG_{U,W}(\Lambda)
=\frac\Lambda{(\Lambda\cap U)+(\Lambda\cap W)}.
$$
If $D\in\End(V)$ satisfies $V[D]=V[D^2]$,
then $V=DV\oplus V[D]$ and we write
$$
\SG_D(\Lambda)=\SG_{DV, V[D]}(\Lambda).
$$
\end{definition}

\begin{example}
Consider $\Lambda=\Z^2\subset \Q^2=V$ and $F=\smallmatrix{0}{1}{1}{0}$, the reflection in the $x=y$ line, and set $D=F\!-\!1$.
The picture is the following:

{\centering{\includegraphics[width=\textwidth]{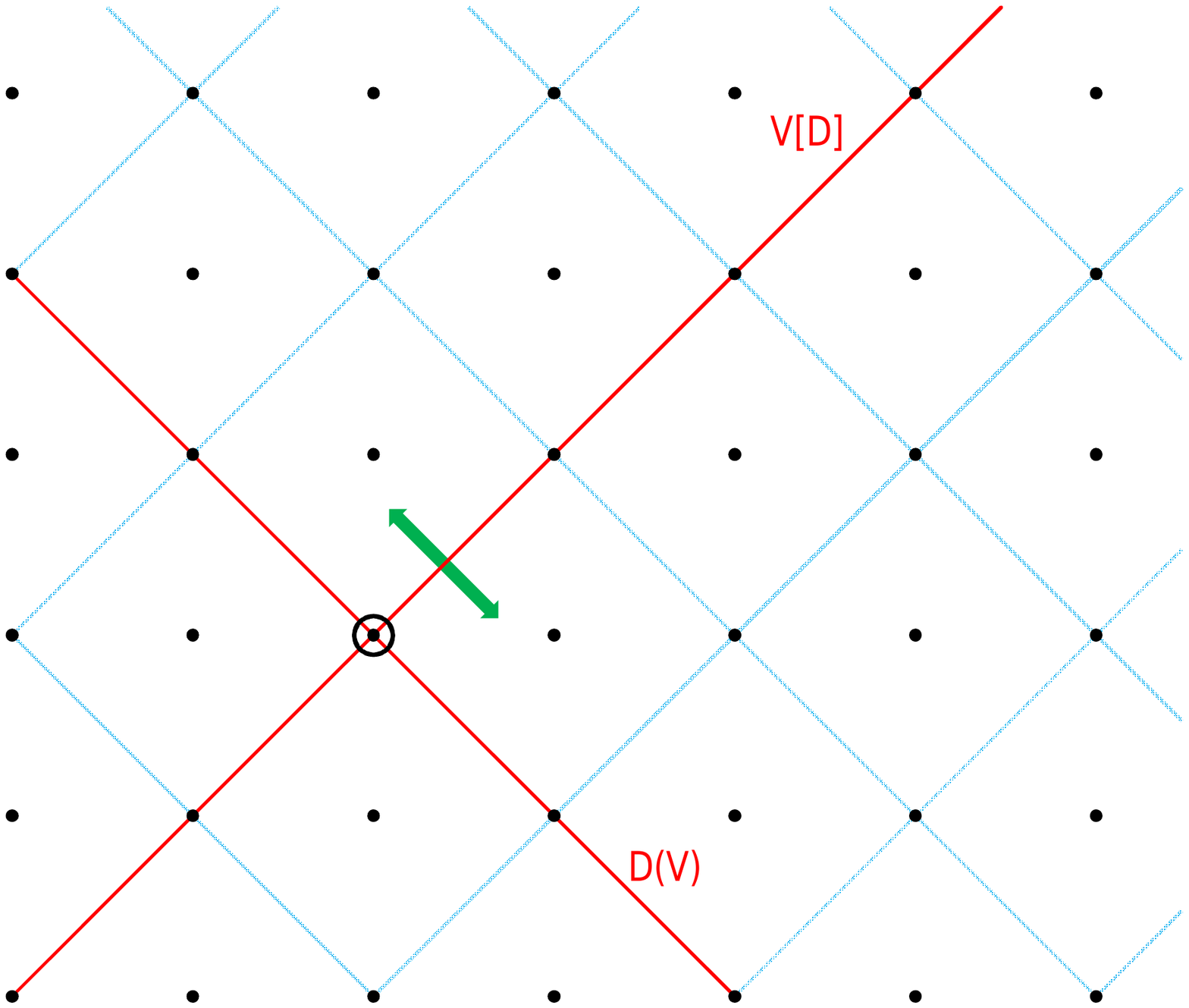}}}

Here the action of $F$ is indicated by the diagonal green arrow. The kernel of~$D$ is $W\!=\!V[D]\!=\!V^F$ is the $x\!=\!y$ line, and its image $U=DV$ is the $x\!=\!-y$ line; these are indicated in red. We clearly have $V=U\oplus W$. However, on the level of lattices this is no longer a direct sum: the blue lines serve to identify the lattice $(\Lambda\cap U)+(\Lambda\cap W)$, and we visibly have the quotient group $\SG_D(\Lambda)\iso C_2$.
\end{example}

\begin{lemma}\label{defnlem:sepgp}
With notation as in the Definition, $\SG_{U,W}(\Lambda)$ is a finite group and
$$
\SG_{U,W}(\Lambda)
\cong\frac{\pi_U\Lambda}{\Lambda\cap U}
\cong\frac{\pi_W\Lambda}{\Lambda\cap W}
\cong\frac{\pi_U\Lambda+\pi_W\Lambda}\Lambda
=\frac{(\Lambda+U)\cap(\Lambda+W)}\Lambda.
$$
\end{lemma}

\begin{proof}
To prove the first isomorphism, simply note
that $\pi_U:\Lambda\rightarrow\frac{\pi_U\Lambda}{\Lambda\cap U}$ is
surjective with kernel
$(\Lambda\cap U)+(\Lambda\cap\ker\pi_U)=(\Lambda\cap U)+(\Lambda\cap W)$.
The second isomorphism follows similarly.

To prove the final isomorphism, consider the map $\pi_U:\Lambda\rightarrow\frac{\pi_U\Lambda+\pi_W\Lambda}\Lambda$. Its kernel clearly contains $(\Lambda\cap U)+(\Lambda\cap W)$, and conversely if $x\in\Lambda$ lies in the kernel then $\pi_U(x)\in\Lambda\cap U$ so $x=\pi_U(x)+(x-\pi_U(x))\in(\Lambda\cap U)+(\Lambda\cap W)$.
Moreover, this map is surjective: its image clearly contains all of $\pi_U\Lambda\) mod \(\Lambda$ and, since $\pi_U(x)\equiv-\pi_W(x)$ mod $\Lambda$, contains all of $\pi_W\Lambda$ mod $\Lambda$ also.
Thus by the first isomorphism theorem, this factors as an isomorphism $\SG_{U,W}(\Lambda)\cong\frac{\pi_U\Lambda+\pi_W\Lambda}\Lambda$.

For the equality, note that a generic element $u+w\in V$ with
$u\in U$ and $w\in W$ lies in $\Lambda+U$ if and only if $w\in \pi_W\Lambda$,
and similarly it lies in $\Lambda+W$ if and only if $u\in\pi_U\Lambda$.
Hence $(\Lambda+U)\cap(\Lambda+W)=\pi_U\Lambda+\pi_W\Lambda$, as required.

Finally $\frac{\Lambda}{((\Lambda\cap U)+(\Lambda\cap W))}$ is clearly
finitely-generated and $\frac{((\Lambda+U)\cap(\Lambda+W))}{\Lambda}$ is
clearly a torsion group, so $\SG_{U,W}(\Lambda)$ is finite.
\end{proof}

\begin{remark}\label{rmk:defnlem:sepgp}
If $\Lambda$ is a lattice with $D\in\End(\Lambda)$, then the separation
group $\SG_D(\Lambda)$ (defined inside $V=\Lambda\otimes\mathbb Q$) is
an intrinsic property of $\Lambda$, not of the vector space it is
embedded in. Also, in this circumstance, the projections
$\pi_{V[D]},\pi_{DV}$ are $D$-equivariant, so the isomorphisms in
Lemma \ref{defnlem:sepgp} are $\mathbb Z[D]$-module isomorphisms. In
fact, we will see in Corollary \ref{cor:sepgpbehaviour} that the $D$
action on $\SG_D(\Lambda)$ is zero.
\end{remark}

\begin{lemma}\label{prop:snf}
Let $V$ be an $m$-dimensional $\Q$-vector space, $D\in\End(V)$
with $V[D] = V[D^2]$, and $\Lambda\subset V$ a full-rank $D$-stable
lattice. Then we have the (abelian group) isomorphism
$$
  \frac{D^{-1}\Lambda}{\Lambda}
  \iso
  \bigoplus_{i=1}^k\left(C_{d_i}\right)\oplus
                   \left(\mathbb Q/\mathbb Z\right)^{m-k}
$$
where $d_1,\dots,d_k$ are the non-zero Smith normal form entries of $D|_\Lambda$.
\end{lemma}

\begin{proof}
By properties of Smith normal form 
we may pick $\Z$-bases $u_i,v_i$ of $\Lambda$ such that $D u_i=d_i v_i$.
Now $D(\sum a_i u_i)\in\Lambda$ if and only if $d_ia_i\in\mathbb Z$ for every $i$.
Thus
$$
  D^{-1}\Lambda=
  \bigoplus_{i=1}^k\left(\frac1{d_i}\mathbb Zu_i\right)
   \oplus\bigoplus_{i=k+1}^m\left(\mathbb Qu_i\right)
$$
and thus the quotient of $D^{-1}\Lambda$ by $\Lambda=\bigoplus\mathbb Zu_i$
is as claimed.
\end{proof}

\begin{proposition}\label{prop:orderofDinverse}
Let $V$ be an $m$-dimensional $\Q$-vector space, $D\in\End(V)$
with $V[D] = V[D^2]$, and $\Lambda\subset V$ a full-rank $D$-stable
lattice. Write the characteristic polynomial of $D$ as $\pm t^rp(t)$
for some $r\ge 0$, and with $p(0)>0$. Then
$$
\left|\frac{D^{-1}\Lambda}{\Lambda+V[D]}\right|=\frac{p(0)}{|\SG_D(\Lambda)|}.
$$
If $D$ is non-degenerate, then $\left|\frac{D^{-1}\Lambda}{\Lambda}\right|=|\det D|$.
\end{proposition}

\begin{proof}
In the non-degenerate case, Lemma \ref{prop:snf} gives $\left|\frac{D^{-1}\Lambda}{\Lambda}\right|=\prod d_i=|\det D|$.

For the general case, write $U=V[D]$ and $W=DV$. By assumption on $D$,
$V=U\oplus W$, and we can let $\pi$ be the corresponding projection onto $W$.
Then
\begin{align*}
\frac{D^{-1}\Lambda}{\Lambda+V[D]}
= \frac{D^{-1}(\Lambda\cap W)}{\Lambda+U} 
&= \frac{(D|_W)^{-1}(\Lambda\cap W)\oplus U}{\pi\Lambda\oplus U} \\
&\cong \frac{(D|_W)^{-1}(\Lambda\cap W)}{\pi\Lambda} \\
&\cong \frac{\left((D|_W)^{-1}(\Lambda\cap W)\right)/(\Lambda\cap W)}
     {\pi\Lambda/(\Lambda\cap W)}
\end{align*}
But $\Lambda\cap W$ is a full-rank lattice in $W$, and $D|_W$ is
non-degenerate with characteristic polynomial $\pm p(t)$, so
(by the non-degenerate case) the numerator of the quotient has size $p(0)$.
The denominator has size $|\SG_D(\Lambda)|$ by Lemma \ref{defnlem:sepgp},
which completes the proof.
\end{proof}


\begin{corollary}\label{cor:sepgpbehaviour}
Let $V$ be a $d$-dimensional $\Q$-vector space and suppose \hbox{$D\in\End(V)$} satisfies $V[D] = V[D^2]$.
Write the characteristic polynomial of $D$ as $\pm t^rp(t)$ and
the minimal polynomial as either $\pm tp_0(t)$ or, if $r=0$, as $\pm p_0(t)$,
with the signs chosen so that $p(0),p_0(0)>0$.

Then for every full-rank $D$-stable lattice $\Lambda\subset V$,
$\SG_D(\Lambda)$ has order dividing $p(0)$, exponent dividing $p_0(0)$,
and is generated by at most $\min(r,d-r)$ elements. The induced action
of $D$ on $\SG_D(\Lambda)$ is the zero map.
\end{corollary}

\begin{proof}
We have from Proposition \ref{prop:orderofDinverse}
that $p(0)/|\SG_D(\Lambda)|$ is the order of a particular group, so it
is an integer. Thus $|\SG_D(\Lambda)|$ divides $p(0)$.

Moreover the $\mathbb Z[D]$-module isomorphisms in Lemma \ref{defnlem:sepgp} realise $\SG_D(\Lambda)$ both as a subquotient of $V[D]$ and of $DV$, so it has both $D$ and $p_0(D)$ in its annihilator.  Since $D$ stabilises $\Lambda$, it has a monic integral minimal polynomial, so that $p_0(0)$ is a $\Z[D]$-linear combination of $D$ and $p_0(D)$. Thus both $D$ and $p_0(0)$ are in the annihilator of $\SG_D(\Lambda)$.

Finally, by Lemma \ref{defnlem:sepgp} 
$\SG_D(\Lambda)$ is a quotient both of $\pi_{V[D]}\Lambda$
and of $\pi_{DV}\Lambda$, which are
finitely generated subgroups of an $r$- and a $(d\!-\!r)$-dimensional
$\Q$-vector space, respectively. Hence $\SG_D(\Lambda)$ is generated
by at most $\min(r,d-r)$ elements, as claimed.
\end{proof}


\subsection{The group $\B_{\Lambda,\Lambda'}$ and $D$-torsion of lattice-quotients}\label{ss:bettsgp}

In this section we introduce the group $\B=\B_{\Lambda,\Lambda'}$, establish its basic properties and prove the main formulae of Theorem \ref{thm:exactformulafixpoint}.

\begin{definition}[$\B_{\Lambda,\Lambda'}$]\label{defn:bettsgp}
Suppose that $\Lambda$ is a lattice, $D\in\End(\Lambda)$ an endomorphism 
satisfying $\Lambda[D]=\Lambda[D^2]$, and that $\Lambda'\subseteq\Lambda$ is a
$D$-stable sublattice of maximal rank. We extend $D$ to a $\Q$-endomorphism
of $V=\Lambda\otimes\Q$ and define 
$$
\B_{\Lambda,\Lambda'}=\frac{\Lambda+D^{-1}\Lambda'}{\Lambda+V[D]}.
$$
\end{definition}


\begin{example}
Consider $\Lambda=\Lambda'=\Z^2\subset \Q^2=V$ and $F=\smallmatrix{-1}{0}{0}{1}$,
the reflection in the $y$-axis, and set $D=F\!-\!1$.
The picture is the following:


\smallskip

\hskip 2cm {\centering{\includegraphics[width=8.5cm]{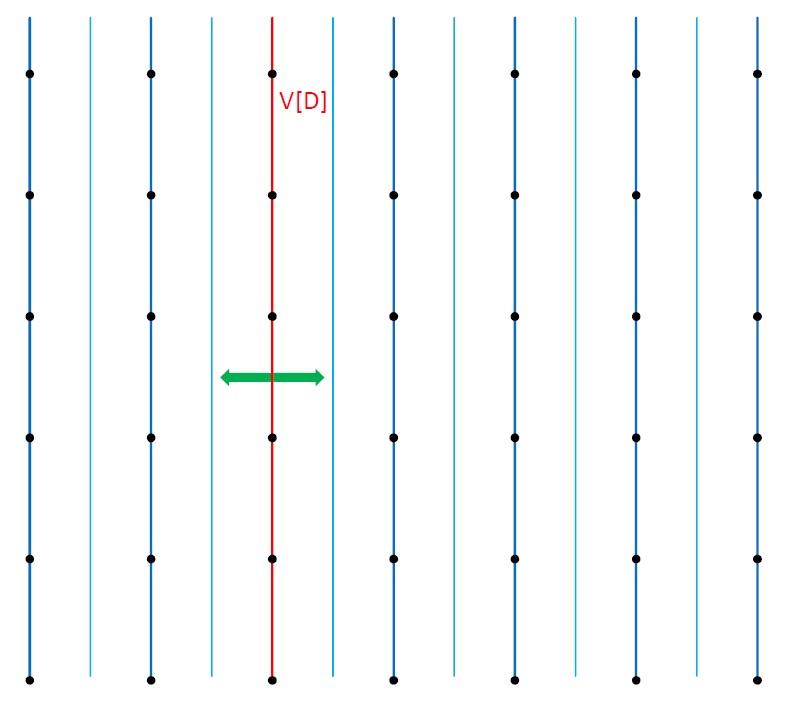}}}


\noindent Here the action of $F$ is indicated by the horizontal green arrow. The kernel of $D$ is $V[D]=V^F=D^{-1}(0)$ is the $y$-axis, indicated in red.
Thus the group $\Lambda+V[D]$ is the set of vertical lines passing through the points of $\Lambda$ (the dark vertical lines in the diagram). The preimage $D^{-1}\Lambda$ is the set of vertical lines with integer or half-integer $x$-coordinates (all vertical lines in the diagram), so the quotient group $\B_{\Lambda,\Lambda}$ is visibly $C_2$.
\end{example}

\begin{proposition}\label{prop:bettsgpbehaviour}
Under the conditions in Definition \ref{defn:bettsgp},
$\B_{\Lambda,\Lambda'}$ is a finite abelian group on at most $(d\!-\!r)$
generators, where $d$ is the rank of $\Lambda$ and $r$ is the nullity of $D$.
Writing the minimal polynomial of $D$ over $V$
as either $\pm tp_0(t)$ or $\pm p_0(t)$, with $p_0(0)>0$,
the exponent of $\B_{\Lambda,\Lambda'}$ divides $p_0(0)$.
The induced action of $D$ on $\B_{\Lambda,\Lambda'}$ is
the zero map.
\end{proposition}

\begin{proof}
Since $D(\Lambda+D^{-1}\Lambda')=D\Lambda+DD^{-1}\Lambda'\subseteq\Lambda$,
we have $D\in\Ann(\B_{\Lambda,\Lambda'})$. 
Moreover $\B_{\Lambda,\Lambda'}$ is evidently a subquotient
of $V/V[D]$, so has $p_0(D)$ in its annihilator. $D$ is
an integral endomorphism (it preserves a lattice), so $p_0(0)$ is
a $\Z[D]$-linear combination
of $D$ and $p_0(D)$, and hence lies in $\Ann(\B_{\Lambda,\Lambda'})$.

Finally, $D^{-1}\Lambda' = (DV\cap D^{-1}\Lambda')\oplus V[D]= (D|_{DV})^{-1}(\Lambda'\cap DV)\oplus V[D]$.
As $D$ is invertible on $DV$, $(D|_{DV})^{-1}(\Lambda'\cap DV)$ is a lattice of rank $d-r$, and hence so is $D^{-1}\Lambda'/V[D]$. 
Thus $\B_{\Lambda,\Lambda'}$ can be generated by $(d\!-\!r)$ elements, and, being a torsion group, must be finite.
\end{proof}

\begin{proposition}\label{prop:bettsgpvaryLambda'}
Under the conditions in Definition \ref{defn:bettsgp}, we have for
all $e\in\N$
$$
  \B_{\Lambda,e\Lambda'} = e\>\B_{\Lambda,\Lambda'}
  \cong\frac{\B_{\Lambda,\Lambda'}}{\B_{\Lambda,\Lambda'}[e]}.
$$
\end{proposition}

\begin{proof}
Working modulo $\Lambda + V[D]$, we have the equality of sets
$$
 \Lambda+D^{-1}e\Lambda' \equiv 
D^{-1}e\Lambda' \equiv 
eD^{-1}\Lambda' \equiv 
 e(\Lambda+D^{-1}\Lambda'),
$$
so $\B_{\Lambda,e\Lambda'} = e\>\B_{\Lambda,\Lambda'}$.
%
%
%
The isomorphism follows from the first isomorphism theorem applied to the multiplication-by-$e$ map on $\B_{\Lambda,\Lambda'}$.
\end{proof}

%

\begin{remark}\label{rmk:Bcoprimef}
In \S\ref{s:tam} we will be working in the setting that $F$ is an automorphism of finite order of $\Lambda$ that preserves $\Lambda'$, and $D=F\!-\!1$.
In this case $\B_{\Lambda,\Lambda'}$ and $\SG_D(\Lambda)$ depend only on the group generated by $F$ and not on the choice of the generator. In other words, if $f$ is coprime to the order of $F$ and $\tilde D=F^f\!-\!1$, then the corresponding group $\tilde\B_{\Lambda,\Lambda'}$ is equal to $\B_{\Lambda,\Lambda'}$, and similarly for $\SG_D(\Lambda)$. Indeed, $\tilde D=F^f\!-\!1$ is clearly divisible by $D=F\!-\!1$ in $\Z[F]$ and vice versa (as $f$ is coprime to the order of $F$), and hence $\tilde D^{-1}\Lambda'=D^{-1}\Lambda'$, $V[\tilde D]= V[D]$ and $\tilde DV= DV$.
\end{remark}

\begin{theorem}\label{thm:exactformulatorsionnoe}
Let $\Lambda$ be a lattice, $D\in\End(\Lambda)$
with $\Lambda[D]=\Lambda[D^2]$, and $\Lambda'\subseteq\Lambda$ a
$D$-stable sublattice of maximal rank.
Write the characteristic polynomial of $D$ over $\Lambda$
as $\pm t^rp(t)$, where $p(0)>0$. Then
$$
\left|\left(\frac\Lambda{\Lambda'}\right)[D]\right|=
|\B_{\Lambda,\Lambda'}|^{-1}\cdot
\left|\frac{\Lambda[ D]}{\Lambda'[ D]}\right|\cdot
\frac{p(0)}{|\SG_D(\Lambda')|}.
$$
\end{theorem}

\begin{proof}
Let $V=\Lambda\otimes_{\Z}\Q$, and consider the following diagram
of $\Z$-submodules of $V$, where all the arrows are inclusion maps:

\begin{center}
\includegraphics[width=11cm]{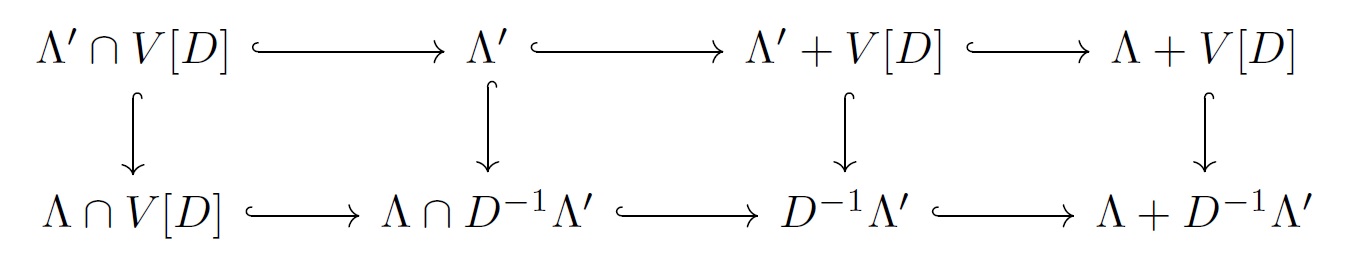}
\end{center}

Now since $\Lambda\cap V[D]=\Lambda[D]$ and similarly for $\Lambda'$,
and $(\Lambda\cap D^{-1}\Lambda')/\Lambda'=\left(\frac{\Lambda}{\Lambda'}\right)[ D]$, this induces a sequence of maps on the quotients
$$
\begin{CD}
0 @>>> \frac{\Lambda[D]}{\Lambda'[D]} @>\alpha>> \left(\frac{\Lambda}{\Lambda'}\right)[D] @>\beta>> \frac{D^{-1}\Lambda'}{\Lambda'+V[D]} @>\gamma>> \B_{\Lambda,\Lambda'} @>>> 0.
\end{CD}
$$

We will show that this sequence is exact by computing the kernel and
image of each map, pulled back to $V$. We make extensive use of the
modular identity for $\Z$-submodules of $V$.

The kernel of $\alpha$ corresponds to
$\Lambda\cap V[D]\cap\Lambda'=\Lambda'\cap V[D]$, so $\alpha$ is injective.
Moreover its image corresponds to $\Lambda'+(V[D]\cap\Lambda)$.

The kernel of $\beta$ corresponds to
$$\Lambda\cap D^{-1}\Lambda'\cap(\Lambda'+V[D])\>=\>\Lambda\cap(V[D]+\Lambda')\>=\>(\Lambda\cap V[D])+\Lambda'.$$
Its image corresponds to
$(\Lambda\cap D^{-1}\Lambda')+\Lambda'+V[D]=
 (D^{-1}\Lambda'\cap\Lambda)+V[ D]$.

The kernel of $\gamma$ corresponds to
$D^{-1}\Lambda'\cap(\Lambda+V[D])=(D^{-1}\Lambda'\cap\Lambda)+V[D]$.
Its image corresponds to $D^{-1}\Lambda'+\Lambda+V[D]=\Lambda+D^{-1}\Lambda'$,
so $\gamma$ is surjective.

Thus the sequence is exact. As all the terms are finite, this ensures
that
$$
\left|\frac{\Lambda[D]}{\Lambda'[D]}\right|\cdot
\left|\frac{D^{-1}\Lambda'}{\Lambda'+V[D]}\right| =
\left|\left(\frac{\Lambda}{\Lambda'}\right)[D]\right|\cdot
|\B_{\Lambda,\Lambda'}|.
$$
Proposition \ref{prop:orderofDinverse} completes the proof.
\end{proof}

\begin{corollary}[to the proof]
\label{cor:BTdivP}
Let $\Lambda$ be a lattice, $D\in\End(\Lambda)$ an endomorphism
satisfying $\Lambda[D]=\Lambda[D^2]$, and $\Lambda'\subseteq\Lambda$ a
$D$-stable sublattice of maximal rank.
Write the characteristic polynomial of $D$ over $\Lambda$
as $\pm t^rp(t)$, where $p(0)>0$. Then
$$
|\B_{\Lambda,\Lambda'}|\cdot|\SG_D(\Lambda')|
\qquad {\text{divides}} \qquad p(0).
$$
\end{corollary}

\begin{proof}
In the proof of Theorem \ref{thm:exactformulatorsionnoe} we constructed
a surjection
$\frac{D^{-1}\Lambda'}{\Lambda'+V[D]}\rightarrow\B_{\Lambda,\Lambda'}$.
By Proposition \ref{prop:orderofDinverse}, the first group has order
$\frac{p(0)}{|\SG_D(\Lambda')|}$, which gives the result.
\end{proof}

\begin{corollary}\label{cor:exactformulatorsion}
Let $\Lambda$ be a lattice, $D\in\End(\Lambda)$
with $\Lambda[D]=\Lambda[D^2]$, and $\Lambda'\subseteq\Lambda$ a
$D$-stable sublattice of maximal rank.
Write the characteristic polynomial of $D$ over $\Lambda$
as $\pm t^rp(t)$, where $p(0)>0$. Then for all $e\in\N$
$$
\left|\left(\frac\Lambda{e\Lambda'}\right)[D]\right|=
\frac{|\B_{\Lambda,\Lambda'}[e]|}{|\B_{\Lambda,\Lambda'}|}\cdot
\left|\frac{\Lambda[D]}{\Lambda'[D]}\right|\cdot
\frac{p(0)}{|\SG_D(\Lambda')|}\cdot e^r.
$$
\end{corollary}

\begin{proof}
Combine Theorem \ref{thm:exactformulatorsionnoe}
with Proposition \ref{prop:bettsgpvaryLambda'}. Note
that $\SG_D(e\Lambda')\cong\SG_D(\Lambda')$
and $|\Lambda'[D]/e\Lambda'[D]|=e^{\rk\Lambda'[ D]}=e^r$.
\end{proof}

\begin{corollary}\label{cor:exactformulatorsionviaF}
Let $\Lambda$ be a lattice, $F\in\Aut(\Lambda)$ an automorphism of finite order, and
$\Lambda'\subseteq\Lambda$ an $F$-stable sublattice of maximal rank.
Write the characteristic polynomial of $F$ over $\Lambda$
as $\pm (t-1)^rp(t)$, where $p(1)>0$, and let $D=F\!-\!1$. Then for all $e\in\N$
$$
\left|\left(\frac\Lambda{e\Lambda'}\right)^F\right|=
\frac{|\B_{\Lambda,\Lambda'}[e]|}{|\B_{\Lambda,\Lambda'}|}\cdot
\left|\frac{\Lambda^F}{\Lambda'^F}\right|\cdot
\frac{p(1)}{|\SG_D(\Lambda')|}\cdot e^r.
$$
\end{corollary}

\begin{proof}
As $F$ has finite order, both its action and the action of $D$ on $\Lambda\otimes_\Z\C$ are diagonalisable, so that in particular $\Lambda[D]=\Lambda[D^2]$. The result now follows from Corollary \ref{cor:exactformulatorsion}.
\end{proof}


\subsection{Lattices with a pairing}\label{ss:pairing}

In this section we generalise the notion of the lattice-dual to a wider
class of $\Z$-submodules of a finite-dimensional $\Q$-vector space~$V$.
The theory is developed relative to a non-degenerate $\Q$-valued pairing
on $V$, but it would be just as straightforward to formulate the theory in
terms of the vector space dual $V^*$ instead.
We will use it to derive further properties of
$\B_{\Lambda,\Lambda'}$ and $\SG(\Lambda')$ in the case when
$\Lambda'=\Lambda^\vee$ is the dual lattice with respect to a suitable
pairing.

Recall that we use the $^\vee$ operator to indicate the dual with respect to a $\Q/\Z$-pairing, so
when $V$ is carries non-degenerate symmetric bilinear form \hbox{$(\cdot,\cdot)\!:\!V\!\!\times\!\! V\!\rightarrow\!\Q$} we write
$
 S^\vee=\{v\in V|\forall s\in S:(v,s)\in \Z \}
$
for a subset $S\subseteq V$ (see Definition \ref{def:dual}).

\begin{notation}
For a finite-dimensional $\Q$-vector space $V$, we will write $\L(V)$
for the set of all subsets of $V$ of the form $M=U+\Lambda$, where $U$
is a subspace and $\Lambda$ a finitely generated $\Z$-submodule.

We write $\dim_-(M)=\dim U$ for the dimension of the largest vector
subspace of $M$, and $\dim_+(M)$ for the dimension of the smallest
subspace containing $M$, i.e.\ the dimension of the $\Q$-span of $M$.
We also define $\codim_\pm(M)$ to be $\dim V-\dim_\pm(M)$.
Note that $\Lambda$ is a full-rank lattice in $V$ if and only if
$\dim_-(\Lambda)=\codim_+(\Lambda)=0$.
\end{notation}

\begin{proposition}\label{prop:L(V)}
If $V$ is a finite-dimensional $\Q$-vector space then
\begin{itemize}
 \item[(a)] $\L(V)$ is a modular poset-lattice with respect to $(+,\cap)$,
 \item[(b)] if $\alpha\in\End(V)$ and $M\in\L(V)$, then $\alpha M,\alpha^{-1}M\in\L(V)$.
\end{itemize}
Suppose $V$ carries a
non-degenerate symmetric pairing $(\cdot,\cdot):V\times V\rightarrow\Q$.
Then
\begin{itemize}
 \item[(c)] $\cdot^{\myperp}:\L(V)\rightarrow\L(V)$ is a self-inverse
       order-reversing isomorphism of poset-lattices;
 \item[(d)] for all $M\in\L(V)$ and $\alpha\in\End(V)$ we have
       $(\alpha M)^{\myperp}=(\alpha^*)^{-1}M^{\myperp}$ and
       $(\alpha^{-1}M)^{\myperp}=\alpha^*M^{\myperp}$, where
       $\cdot^*$ denotes the adjoint with respect to $(\cdot,\cdot)$;
 \item[(e)] for all $M\in\L(V)$, $\dim_+(M^{\myperp}) = \codim_-(M)$ and
       $\dim_-(M^{\myperp})=\codim_+(M)$. In particular, if $\Lambda$
       is a full-rank lattice in $V$, then so is $\Lambda^{\myperp}$;
 \item[(f)] if $N\subseteq M$ are both in $\L(V)$ such that $M/N$ is finite, then
       $M/N\iso N^{\myperp}/M^{\myperp}$.
\end{itemize}
\end{proposition}

\begin{proof}
For the proofs of (a) and (b) we may also, without loss of generality, assume
that $V$ carries some non-degenerate symmetric $\Q$-valued pairing.

\noindent \textbf{Points (a) and (c):}
Suppose $M=U+\Lambda\in\mathcal L(V)$. Note that $M/U$ is finitely generated,
so we may pick some $v_j\in M$ such that $v_j+U$ form a $\Z$-basis of $M/U$.
Now let $u_i$ be a $\Q$-basis of $U$, and extend the independent vectors
$u_i,v_j$ to a basis $u_i,v_j,w_k$ of $V$. Then we have
$$
  M=\bigoplus_i\Q u_i\oplus\bigoplus_j\Z v_j\oplus\bigoplus_k 0w_k.
$$

Now let $u_i^\myperp,v_j^\myperp,w_k^\myperp$ be the basis of $V$ dual
to $u_i,v_j,w_k$. The generic element
$\sum_ia'_iu_i^\myperp+\sum_jb'_jv_j^\myperp+\sum_kc'_kw_k^\myperp$
of $V$ lies in $M^\myperp$ precisely
when $\sum_ia_ia'_i+\sum_jb_jb'_j\in\Z$ for all $a_i\in\Q$ and $b_j\in\Z$,
i.e.\ exactly when $a'_i=0$ and $b'_j\in\Z$ for all $i,j$. Thus
$$
  M^\myperp=\bigoplus_i 0u_i^\myperp\oplus\bigoplus_j\Z v_j^\myperp
            \oplus\bigoplus_k\Q w_k^\myperp.
$$
In particular, $M^{\myperp}\in\L(V)$. Also, as $u_i,v_j,w_k$ is dual
to $u_i^\myperp,v_j^\myperp,w_k^\myperp$ it follows that $M^{\myperp\myperp}=M$.

Now it is clear that $\L(V)$ is closed under $+$ and that for
$M,N\in\L(V)$ we have
$$
  (M+N)^\myperp=(M\cup N)^\myperp=M^\myperp\cap N^\myperp.
$$
In particular, $M\cap N=(M^\myperp+N^\myperp)^\myperp\in\L(V)$,
completing the proof of (a).

Moreover
 $(M\cap N)^\myperp=(M^\myperp+N^\myperp)^{\myperp\myperp}=M^\myperp+N^\myperp$,
completing the proof of (c).

\noindent \textbf{Points (b) and (d):}
Since $\alpha(U+\Lambda)=\alpha U+\alpha\Lambda$, we have $\alpha M\in\L(V)$. Moreover, $(x,\alpha y)\in\Z$ precisely
when $(\alpha^*x,y)\in\Z$, so $(\alpha M)^\myperp=(\alpha^*)^{-1}M^\myperp$.
In particular $\alpha^{-1}M=(\alpha^*M^\myperp)^\myperp\in\L(V)$, completing
the proof of (b). Also, taking $\cdot^\myperp$ of each side
gives $(\alpha^{-1}M)^\myperp=\alpha^*M^\myperp$, which completes (d).

\noindent \textbf{Point (e):}
Using the coordinate description of $\cdot^\myperp$ in our proof of (a) and (c),
note that $\dim_+(M)=\#\{u_i\}+\#\{v_j\}$ and
that $\dim_-(M^\myperp)=\#\{w_k^\myperp\}=\#\{w_k\}$.
Thus $\dim_+(M)+\dim_-(M^{\myperp})=\#\{u_i,v_j,w_k\}=\dim V$,
so $\dim_+(M)=\codim_-(M^\myperp)$. The other statement follows by
replacing $M$ with $M^\myperp$.

\noindent \textbf{Point (f):}
Consider the pairing $(\cdot,\cdot)$ reduced mod $\Z$ and restricted
to $M\times N^{\myperp}$. Its left-kernel is $M\cap N^{\myperp\myperp}=N$
and its right-kernel is $N^{\myperp}\cap M^{\myperp}=M^{\myperp}$ so
it passes to a non-degenerate pairing on
$\frac MN\times\frac{N^{\myperp}}{M^{\myperp}}$, taking values
in $\Q/\Z$.
Since $M/N$ is finite, the pairing must be perfect, and hence
$N^{\myperp}/M^{\myperp}\cong\Hom(M/N,\Q/\Z)\iso M/N$.
\end{proof}

\begin{lemma}\label{lem:dlvee}
Let $V$ be a finite-dimensional $\Q$-vector space, $F\!=\!1\!+\!D$ an automorphism of $V$, 
and $(\cdot,\cdot):V\times V\rightarrow\Q$ an $F$-invariant non-degenerate symmetric pairing on $V$.
Then
$$
 (DV)^\vee = V[D].
$$
If $\Lambda$ is an $F$-invariant full-rank lattice in $V$, then
$$
 (D\Lambda)^\vee = D^{-1}\Lambda^\vee.
$$
\end{lemma}
\begin{proof}
Observe first that $D^*=(F-1)^*=F^{-1}-1=(1-F)F^{-1}=-DF^{-1}$.
Hence, by Proposition \ref{prop:L(V)}(d),
$$
  (DV)^\vee = (D^*)^{-1}V^\vee = D^{-1}\{0\}= V[D],
$$
and similarly for $(D\Lambda)^\vee$. (Note that $F\Lambda=\Lambda$, since $F\Lambda\subseteq\Lambda$ by assumption and $F$ has determinant $\pm1$, being orthogonal with respect to $(\cdot,\cdot)$).
\end{proof}

\begin{theorem}\label{lem:perppreservessepgps}
Let $V$ be a finite-dimensional $\Q$-vector space endowed with a
non-degenerate symmetric pairing $(\cdot,\cdot):V\times V\rightarrow\Q$,
decomposing into orthogonal subspaces as $U\oplus W$.
Suppose $\Lambda\subseteq V$ is a full-rank sublattice.
Then
$$
  \SG_{U,W}(\Lambda)\simeq\SG_{U,W}\left(\Lambda^\myperp\right).
$$
If $F\!=\!1\!+\!D$ is an automorphism of $\Lambda$ that preserves $(\cdot,\cdot)$, then
$\SG_D(\Lambda)=\SG_D(\Lambda^\vee)$.
\end{theorem}

\begin{proof}
Note that $U^\myperp=W$ and vice-versa. Using
Proposition \ref{prop:L(V)} we deduce that
\begin{align*}
\SG_{U,W}(\Lambda^\myperp) &= \frac{\Lambda^\myperp}{(\Lambda^\myperp\cap U)+(\Lambda^\myperp\cap W)} \\
&\simeq\frac{\left((\Lambda^\myperp\cap U)+(\Lambda^\myperp\cap W)\right)^\myperp}{\Lambda^{\myperp\myperp}} \\
&\cong\frac{(\Lambda^{\myperp\myperp}+U^\myperp)\cap(\Lambda^{\myperp\myperp}+W^\myperp)}\Lambda \\
&\cong\frac{(\Lambda+W)\cap(\Lambda+U)}\Lambda\cong\SG_{U,W}(\Lambda),
\end{align*}
where the last isomorphism comes from Lemma \ref{defnlem:sepgp}.

The final claim now follows from Lemma \ref{lem:dlvee}.
\end{proof}

\begin{lemma}\label{cor:LambdaperpisFstable}
Let $V$ be a finite-dimensional $\Q$-vector space,
$F$ an automorphism of $V$,
and $(\cdot,\cdot):V\times V\rightarrow\Q$ an $F$-invariant
non-degenerate symmetric pairing. If $\Lambda\subset V$ is an $F$-stable
lattice, then so is $\Lambda^\vee$.
\end{lemma}

\begin{proof}
Simply note that $F\Lambda^\myperp=(F\Lambda)^\myperp=\Lambda^\myperp$
by Proposition \ref{prop:L(V)}(d). 
\end{proof}

\begin{lemma}\label{lem:alternativebettsgp}
Let $V$ be a finite-dimensional $\Q$-vector space, $F=1\!+\! D$ an
automorphism of $V$ that satisfies $V[D^2]=V[D]$,
and $(\cdot,\cdot):V\times V\rightarrow\Q$ an
$F$-invariant non-degenerate symmetric pairing.
If $\Lambda\subset V$ is an $F$-stable full-rank lattice with $\Lambda^\myperp\subseteq\Lambda$, then
$$
 \B_{\Lambda,\Lambda^\myperp}=
 \frac{\Lambda+D^{-1}\Lambda^\myperp}{\Lambda+V[D]}\iso
 \frac{\Lambda^\myperp\cap DV}{\Lambda^\myperp\cap D\Lambda}.
$$
Both groups have trivial induced $F$ action.
\end{lemma}

\begin{proof}
By Lemma \ref{lem:dlvee}, $D\Lambda=(D^{-1}\Lambda^\myperp)^\myperp$ and
$DV=(V[D])^\myperp$, so the isomorphism follows from
Proposition \ref{prop:L(V)}(c,f).
By Lemma \ref{cor:LambdaperpisFstable}, $\Lambda^\vee$ is $F$-stable, so
$F$ acts trivially on $\B_{\Lambda,\Lambda^\myperp}$
by Proposition \ref{prop:bettsgpbehaviour}.
In addition, since $D\Lambda^\myperp\subseteq \Lambda^\myperp\cap D\Lambda$, $D$ is
in the annihilator of $(\Lambda^\myperp\cap DV)/(\Lambda^\myperp\cap DV)$,
so $F$ also acts trivially on this quotient as well.
\end{proof}

\begin{remark}
The condition that $\Lambda^\vee\subseteq\Lambda$ in the preceding lemma is equivalent to requiring that the pairing is $\Z$-valued on $\Lambda^\vee$. For our applications, we will start with a $\Z[F]$-lattice $X$ with a non-degenerate $\Z$-valued $F$-equivariant pairing, and will let $\Lambda$ be dual lattice, viewed as an overlattice of $X$ via the pairing. In this setup, we have $\Lambda^\vee=X$, so that the condition $\Lambda^\vee\subseteq\Lambda$ is automatically satisfied.

Notice also that the characteristic polynomials of $F$ on $\Lambda$ and $\Lambda^\vee$ agree, as do the minimal polynomials, so the conditions in Theorem \ref{thm:exactformulafixpoint} can be checked either on $\Lambda$ or on $\Lambda^\vee=X$.
\end{remark}

\begin{corollary}\label{cor:restricting_B}
Let $V$ be a finite-dimensional $\Q$-vector space, $F=1\!+\! D$ an automorphism of $V$ that satisfies $V[D^2]=V[D]$, and $(\cdot,\cdot):V\times V\rightarrow\Q$ an $F$-invariant non-degenerate symmetric pairing. Let $V_0=DV$ be the largest $F$-invariant subspace on which $D$ is non-singular, so that the orthogonal complement of $V_0$ is $V[D]$. Suppose that $\Lambda\subset V$ is an $F$-stable full-rank lattice with $\Lambda^\vee\subseteq\Lambda$, and write $\Lambda_0^\vee\subset V_0$ (resp.\ $\Lambda_0\subset V_0$) for the intersection of $\Lambda^\vee$ with $V_0$ (resp.\ the projection of $\Lambda$ to $V_0$).

Then the restriction of $(\cdot,\cdot)$ to $V_0$ is non-degenerate, $\Lambda_0^\vee\subseteq\Lambda_0$ are dual under this restricted pairing, and\[\B_{\Lambda_0,\Lambda_0^\vee}\simeq\B_{\Lambda,\Lambda^\vee}.\]
\begin{proof}
Since $V_0$ is an orthogonal direct summand of $V$, it inherits non-degeneracy of the pairing. To see that $\Lambda_0^\vee$ and $\Lambda_0$ are dual, note that the dual of $\Lambda_0^\vee$ in all of $V$ is $(\Lambda^\vee\cap V_0)^\vee=\Lambda+V[D]$ by Proposition \ref{prop:L(V)}, and hence its dual inside $V_0$ is $(\Lambda+V[D])\cap V_0=\Lambda_0$. For the final part, note that $DV=DV_0=V_0$ and so $\Lambda_0^\vee\cap DV_0=\Lambda^\vee\cap DV$ and that $\Lambda_0^\vee\cap D\Lambda_0=\Lambda^\vee\cap D\Lambda$ since the projection to $V_0$ is $F$-equivariant. Lemma \ref{lem:alternativebettsgp} then provides the desired isomorphism.
\end{proof}
\end{corollary}

\begin{theorem}\label{thm:bettspairing}
Let $V$ be a finite-dimensional $\Q$-vector space, $F=1\!+\! D$ an
automorphism of $V$ that satisfies $V[D^2]=V[D]$,
and $(\cdot,\cdot):V\times V\rightarrow\Q$ an
$F$-invariant non-degenerate symmetric pairing.
Let $\Lambda\subset V$ be an $F$-stable full-rank lattice satisfying $\Lambda^\myperp\subseteq\Lambda$.
Then $\B_{\Lambda,\Lambda^\vee}$ admits a perfect antisymmetric pairing taking values in $\Q/\Z$.

Explicitly,
$$
  \langle x,y\rangle\equiv(x,(D|_{DV})^{-1}y)\text{ mod }\Z
$$
gives a well-defined pairing
$\langle\cdot,\cdot\rangle:DV\times DV\rightarrow\Q/\Z$ which satisfies:
\begin{enumerate}
  \item[($i$)] $\langle x,y\rangle+\langle y,Fx\rangle=0$ for all $x,y\in DV$,
  \item[($ii$)] $\langle\cdot,\cdot\rangle$ passes to a perfect antisymmetric pairing
  on $\frac{\Lambda^\myperp\cap DV}{\Lambda^\myperp\cap D\Lambda}\iso\B_{\Lambda,\Lambda^\myperp}$.
\end{enumerate}
\end{theorem}

\begin{proof}
By Corollary \ref{cor:restricting_B}, we may assume for simplicity of notation that $D$ is non-singular on $\Lambda$, so $D$ is invertible on $V$ and $\B_{\Lambda,\Lambda^\vee}\simeq\frac{\Lambda^\vee}{\Lambda^\vee\cap D\Lambda}$. Any $x,y\in V$ may be written as $x=Dx_0$, $y=Dy_0$, so that
\begin{align*}
 \langle y,Fx\rangle = (Dy_0,Fx_0) &= (Fy_0,Fx_0) - (y_0,Fx_0) = \\
 &= (x_0,y_0)-(Fx_0,y_0) = - (Dx_0,y_0) = - \langle x,y\rangle,
\end{align*}
which proves (i).

Consider the restriction of $\langle\cdot,\cdot\rangle$ to $\Lambda^\vee$. Since $F$ is an automorphism of $\Lambda^\vee$ by Lemma \ref{cor:LambdaperpisFstable}, identity $(i)$ shows that the left- and right-kernels of $\langle\cdot,\cdot\rangle$ on $\Lambda^\vee$ coincide. Moreover, we see that $y\in\Lambda^\vee$ is in the right-kernel precisely when $D^{-1}y\in(\Lambda^\vee)^\vee=\Lambda$, so that the left- and right-kernels are both equal to $\Lambda^\vee\cap D\Lambda$.

It follows that $\langle\cdot,\cdot\rangle$ passes to a perfect pairing on $\frac{\Lambda^\vee}{\Lambda^\vee\cap D\Lambda}\simeq\B_{\Lambda,\Lambda^\vee}$. Antisymmetry of this pairing follows directly from (i) and the fact that $F$ acts trivially on $\B_{\Lambda,\Lambda^\myperp}$ (Proposition \ref{prop:bettsgpbehaviour}).
\end{proof}


\subsection{Life up to squares}

We now turn to the behaviour of the order of the \hbox{$F^f$-invariants} of $\Lambda/e\Lambda^{\!\vee}$ up to rational squares as $e$ and $f$ vary. We first establish some general results about sizes of $\A[e]$ and $\A^{F^f}$  up to squares:

\begin{theorem}\label{thm:uptosquares}
Let $\mathcal A$ be a finite abelian group that admits a perfect
pairing $\langle\cdot,\cdot\rangle:\mathcal A\times\mathcal A\rightarrow\Q/\Z$.
Let $\sim$ denote equality up to rational squares.

\begin{enumerate}
\item If $\langle,\rangle$ is antisymmetric, then $|\A|$ is either a square or twice a square; it is a square if and only if $\mathcal A$ admits a perfect alternating pairing.
Moreover,
$$
 |\A[e]|\sim
   \begin{cases}
     |\A| & \text{if}\quad 2\div e, \\
        1 & \text{if}\quad 2\nmid e.
   \end{cases}
$$
\item If $F\in\Aut(\A)$ and $\langle,\rangle$ is symmetric and $F$-invariant,
then
$$
  |\A^{F^f}|\sim
    \begin{cases}
      |\A^F| & \text{if}\quad 2\nmid f, \\
      |\A|   & \text{if}\quad 2\div f.
    \end{cases}
$$
\end{enumerate}
\end{theorem}

\begin{lemma}\label{lem:reversible}
Let $\A$ be an abelian group with an automorphism $F$ and
with a $\Q/\Z$-valued $F$-equivariant symmetric perfect
pairing $(\cdot,\cdot)$ on $\A$. Suppose that \hbox{$H\in\Z[F,F^{-1}]$} has $\theta(H)=F^{2s}H$
for some $s\in\Z$, where $\theta$ is the unique automorphism of
$\Z[F,F^{-1}]$ interchanging $F$ and $F^{-1}$. Then for all
$k\in\Z$ the quotient $\A/\A[(F^{2k}-1)H]$ admits a $\Q/\Z$-valued perfect
alternating pairing.
\end{lemma}

\begin{proof}
Since $\theta(F^sH)=F^{-s}F^{2s}H=F^sH$ and
$\A[(F^{2k}-1)H]=\A[(F^{2k}-1)F^sH]$, we may, without loss of generality,
assume that $s=0$.

Note that now $H$ is self-adjoint and that $F^k\!-\! F^{-k}$ is anti-self-adjoint
with respect to $(\cdot,\cdot)$. Thus the pairing on $\A$ given by
$$
 \langle x,y\rangle=((F^k-F^{-k})Hx,y)=-(x,(F^k-F^{-k})Hy)
$$
has left- and right-kernels $\A[(F^k-F^{-k})H]=\A[(F^{2k}-1)H]$,
and hence passes to a perfect pairing on the desired quotient.
To see that it is alternating, note that
\begin{equation*}
  \langle x,x\rangle=(F^kHx,x)-(F^{-k}Hx,x)=(F^kHx,x)-(x,F^kHx)=0.\qedhere
\end{equation*}
\end{proof}

\begin{lemma}\label{prop:idealsannihilate}
Let $R$ be a commutative ring, and $I,J,K$ ideals of $R$ with $I+J=R$. Then for any $R$-module $M$,
$$
  M[IJK]=M[IK]+M[JK].
$$
\end{lemma}

\begin{proof}
Clearly $M[IK]+M[JK]\subseteq M[IJK]$, so it suffices to check that the quotient
$\tilde M=\frac{M[IJK]}{M[IK]+M[JK]}$ is trivial. Since $I.M[IJK]\subseteq M[JK]$
we see that $I\subseteq\Ann(\tilde M)$. Similarly,
$J\subseteq\Ann(\tilde M)$, so $R=I+J\subseteq\Ann(\tilde M)$, and $\tilde M$ is trivial as desired.
\end{proof}

\begin{proof}[Proof of Theorem \ref{thm:uptosquares}]

(1) The fact about the order of $\A$ is standard.

Let $\langle\cdot,\cdot\rangle_{e}$ denote the reduction of
$\langle\cdot,\cdot\rangle$ mod $\frac1e\Z$, for $e\in\N$.
Thus $x\in\A$ is in the left-kernel (equivalently, right-kernel)
of $\langle\cdot,\cdot\rangle_e$ just when
$e\langle x,y\rangle=\langle ex,y\rangle=0$ for all $y\in\A$,
i.e.\ just when $x\in\A[e]$.
Thus $\langle\cdot,\cdot\rangle_e$ passes to a perfect antisymmetric
pairing on $\A/\A[e]$.

When $e$ is odd
$\left|\frac{\A}{\A[1]}\right|/\left|\frac{\A}{\A[e]}\right|=|\A[e]|$
is then either a square or twice a square, but is also the order of a group
with odd exponent; thus $|\A[e]|\sim 1$.
If $e$ is even, then as $2\langle x,x\rangle=0$ for all $x\in\A$, it follows that $\langle x,x\rangle_e=0$ and\ $\langle\cdot,\cdot\rangle_e$ is alternating; hence $|\A[e]|\sim|\A|$.

(2)
Without loss of generality, $f>0$.

If $2\div f$ then, Lemma \ref{lem:reversible} with $H=1,k=\frac f2$ shows
that $\A/\A^{F^f}$ admits a $\Q/\Z$-valued
alternating perfect pairing, and so has square order,
i.e.\ $\smash{\left|\A^{F^f}\right| \sim |\A|}$.

For $2\nmid f$ consider $\A^{F^f}$ and $\A^{F^2}$.
Their intersection is precisely $\A^F$ and by
Lemma \ref{prop:idealsannihilate}
(with $I=(F+1),J=(F^{f-1}+F^{f-2}+\dots+1),K=(F-1)$) their sum
is $\A[(F^2-1)(F^{f-1}+\dots+1)]$.
Thus by the second
isomorphism theorem for groups
$$
  \frac{\A^{F^f}}{\A^F}= \frac{\A^{F^f}}{\A^{F^f}\cap \A^{F^2}}
  \cong \frac{\A[(F^2-1)(F^{f-1}+\dots+1)]}{\A^{F^2}}.
$$
Lemma \ref{lem:reversible} with $k=1,H=F^{f-1}+\dots+1$ shows that
$\A[(F^2-1)(F^{f-1}+\dots+1)]$ has square index in $\A$, as does $\A^{F^2}$
(by the ``$2|f$'' case). Hence $\smash{\left|\A^{F^f}\right|\sim\left|\A^F\right|}$.
\end{proof}

\begin{theorem}\label{thm:uptosquares2}
Let $\Lambda$ be a lattice with an automorphism $F=1+D$,
satisfying $\Lambda[D]=\Lambda[D^2]$.
Suppose $(\cdot,\cdot)$
is a non-degenerate symmetric $F$-invariant pairing with
$\Lambda^\vee\subseteq\Lambda$.
Write the characteristic polynomial of $F$ over $\Lambda$
as $\pm(t-1)^rp(t)$ with $p(1)>0$.
Then for all $e,f\in\N$ 
$$
  \left|\left(\frac\Lambda{e\Lambda^\vee}\right)^{F^f}\right|\sim
    \begin{cases}
       \left|\frac{\Lambda^F}{\Lambda^{\vee F}}\right|\cdot\frac{p(1)}{|\B_{\Lambda,\Lambda^\vee}|\cdot|\SG_D(\Lambda)|}\cdot e^r
        & \text{if }\>2\nmid e,2\nmid f, \\
      \left|\frac{\Lambda^F}{\Lambda^{\vee F}}\right|\cdot\frac{p(1)}{|\SG_D(\Lambda)|}\cdot e^r
        & \text{if }\>2\div e,2\nmid f, \\
      \left|\frac\Lambda{\Lambda^\vee}\right|\cdot e^{\rk\Lambda}
        & \text{if }\>2\div f,
    \end{cases}
$$
where $\sim$ denotes equality up to rational squares.
\end{theorem}


\begin{proof}
The reduction mod $e\Z$ of $(\cdot,\cdot)$ provides a
$\Q/e\Z$-valued $F$-invariant perfect symmetric pairing
on $\Lambda/e\Lambda^\vee$.
By Theorem \ref{thm:uptosquares}(2),
$$
  \left|\left(\frac\Lambda{e\Lambda^\vee}\right)^{F^f}\right|\sim
    \begin{cases}
       \left|\left(\frac\Lambda{e\Lambda^\vee}\right)^{F}\right|
        & \text{if }2\nmid f, \\
       \left|\frac\Lambda{e\Lambda^\vee}\right|
        & \text{if }2\div f,
    \end{cases}
$$
As $|\Lambda/e\Lambda^\vee|=|\Lambda/\Lambda^\vee|.e^{\rk\Lambda}$,
this gives the result when $2\div f$.

Since $\Lambda^\vee$ is $F$-stable (Corollary \ref{cor:LambdaperpisFstable})
we may apply Corollary \ref{cor:exactformulatorsionviaF}:
$$
\left|\left(\frac\Lambda{e\Lambda^\vee}\right)^F\right|
  =\frac{|\B_{\Lambda,\Lambda^\vee}[e]|}{|\B_{\Lambda,\Lambda^\vee}|}.
  \left|\frac{\Lambda^F}{\Lambda^{\vee F}}\right|.\frac{p(1)}{|\SG_D(\Lambda^\vee)|}.e^r.
$$
Note that by Theorem \ref{lem:perppreservessepgps},
$\SG_D(\Lambda)\iso\SG_D(\Lambda^\vee)$.
Finally, by Theorem \ref{thm:bettspairing}, $\B_{\Lambda,\Lambda^\vee}$
possesses a $\Q/\Z$-valued perfect antisymmetric pairing, so
by Theorem \ref{thm:uptosquares}(1),
\begin{equation*}
 \frac{|\B_{\Lambda,\Lambda^\vee}[e]|}{|\B_{\Lambda,\Lambda^\vee}|}\sim
   \begin{cases}
     \frac{1}{|\B_{\Lambda,\Lambda^\vee}|} & \text{if}\quad 2\nmid e, \\
                           1 & \text{if}\quad 2\div e.
   \end{cases}
\qedhere
\end{equation*}
\end{proof}


\medskip

\subsection{$\Z[C_n]$-lattices}
\label{ss:c_n}

We now examine $C_n$-lattices with pairings in some specific cases, computing their invariants $\B$ and $\SG$, and often classifying the possibilities that can occur. The goal of this section is the complete classification of lattice-pairs in dimension $\leq2$.

We begin with an example to illustrate how the above results can be applied in practice. Note throughout this section that for any $C_n$-lattice $\Lambda$ the group action on $\Lambda\otimes_\Z\Q$ is semisimple, so the hypothesis $\Lambda[D]=\Lambda[D^2]$ is automatically satisfied for any generator $F$ of $C_n$ and $D\!=\!F\!-\!1$.

\smallskip

\begin{lemma}
Let $\Lambda$ be an integral representation of $C_n$ with a full-rank subrepresentation $\Lambda'\subseteq\Lambda$, and let $F$ be a generator of $C_n$. Suppose that $\Lambda'$ is a permutation representation corresponding to an action of $C_n$ with orbits of size $m_1,\dots,m_k$. Then $\SG(\Lambda')\simeq\prod_iC_{m_i}$ and $\B_{\Lambda,\Lambda'}$ is trivial.
\begin{proof}
The characteristic polynomial of $F$ is $\prod_i\left(X^{m_i}-1\right)=(X-1)^k\prod_i\left(\frac{X^{m_i}-1}{X-1}\right)$ so that in the notation of Theorem \ref{thm:exactformulafixpoint} we have $p(1)=\prod_im_i$. Since $|\SG(\Lambda')|.|\B_{\Lambda,\Lambda'}|$ divides $p(1)$, triviality of $\B_{\Lambda,\Lambda'}$ will follow from the description of $\SG(\Lambda')$.

To prove this, since the construction of the separation group respects direct sums of $\Z[F]$-lattices, it suffices to consider the case when the $C_n$-set giving rise to $\Lambda'$ is transitive, i.e.\ $\Lambda'\simeq\Z[C_m]$ for some $m$ dividing $n$. Then $\pi=\frac1m(1+F+\dots+F^{m-1})$ is the projection operation from $V=\Lambda'\otimes\Q$ to $V^F$ with kernel $(F-1)V$. It follows that $\pi\Lambda'$ is the $\Z$-span of $\frac1m(1+F+\dots+F^{m-1})$ in $\Q[C_m]\simeq V$, while $\Lambda'\cap V^F$ is the $\Z$-span of $1+F+\dots+F^{m-1}$. It follows by Lemma \ref{defnlem:sepgp} that $\SG(\Lambda')\simeq C_m$, which was what we wanted to prove.
\end{proof}
\end{lemma}

We now turn our attention to classifying the low-dimensional $C_n$-lattice pairs, beginning by examining the cases when $\Lambda\otimes\Q$ is an irreducible $\Q$-representation of $C_n$. For this, we recall first the following basic fact about cyclotomic polynomials:
\begin{lemma}\label{lem:cyc}
$\Phi_k(1)\!=\!p$ if $k\!=\!p^s$ is a prime power; $\Phi_k(1)\!=\!1$ for all other $k\!\neq\! 1$.
\end{lemma}
\begin{proof}
This follows by substituting $X=1$ into the formula $\frac{X^k-1}{X-1}=\prod_{1\neq j|k}\Phi_j(X)$ and using induction on $k$.
\end{proof}

\newpage

\begin{theorem}
\label{thm:c_n}
Let $\Lambda$ be an integral representation of $C_n$ ($n>1$), such that
$\Lambda\otimes\Q$ is the unique faithful irreducible rational representation
of $C_n$. Fix a generator $F$ of $C_n$ and write $D=F-1$ and $p(t)$ for the
minimal polynomial of $F$.

(i) $\Lambda\simeq\a$ for some fractional ideal $\a$ of
$\Q(\zeta_n)$, where $F$ acts on $\Q(\zeta_n)$ by multiplication
by a primitive $n$-th root of unity $\zeta_n$.
Two such lattices are isomorphic as $\Z[C_n]$-representations
if and only if the corresponding fractional ideals are equivalent
in the ideal class group $\Cl_{\mathbb Q(\zeta_n)}$.

(ii) The non-degenerate $C_n$-invariant $\Q$-valued symmetric pairings
on $\Lambda\simeq\a$ are given by
$$
  (x,y)=\tr_{\Q(\zeta_n)/\Q}(\omega^{-1}x\bar y)
$$
for non-zero $\omega\in\Q(\zeta_n)^+$.
With respect to this pairing
$\Lambda^\myperp=\omega\different_n^{-1}\bar{\a}^{-1}$, where
$\different_n^{-1}=\different_{\Q(\zeta_n)/\Q}^{-1}$ is the inverse different
ideal \footnote{i.e. $\different_n^{-1}=\{z\in\Z[\zeta_n]\>|\>\forall w\in\Z[\zeta_n]: \tr_{\Q(\zeta_n)/\Q}(zw)\in\Z\}$}.
In particular $\Lambda^\vee\subseteq\Lambda$ if and only if $\omega\in\mathbf{a\bar a}\different_n$.

(iii) Suppose $\Lambda$ admits such a pairing and that
$\Lambda^\vee\subseteq \Lambda$. Then $\SG_D(\Lambda)$ is trivial; $p(1)=q$ if $n=q^s$ is a prime power, and
$p(1)=1$ otherwise; $\B_{\Lambda,\Lambda^\vee}$ is trivial unless $n=2^s$
and $[\Lambda:\Lambda^\vee]$ is odd, in which case $\B_{\Lambda,\Lambda^\vee}\iso C_2$. In particular
$$
 \left|\left(\frac\Lambda{e\Lambda^\vee}\right)^F\right| =
   \begin{cases}
      q & \text{if}\quad n=q^s, q\text{ an odd prime}, \\
      2 & \text{if}\quad n=2^s \text{ and } [\Lambda\!:\!e\Lambda^\vee] \text{ is even}, \\ 
      1 & \text{otherwise}.
   \end{cases}
$$
\end{theorem}

\begin{proof}
(i) 
$\Lambda\otimes_\Z \Q\cong \Q(\zeta_n)$, this being the unique faithful
irreducible rational representation of $C_n$. The image of $\Lambda$ in
$\Q(\zeta_n)$ is then a $\Z[\zeta_n]$-submodule, and hence a fractional
ideal $\a$ of $\Q(\zeta_n)$. Finally, two such lattices are isomorphic
as $\Z[C_n]$-representations if and only if the corresponding fractional
ideals are isomorphic as $\Z[\zeta_n]$-modules, which is equivalent to them
having the same class in $\Cl_{\Q(\zeta_n)}$ (only possible isomorphisms
are scaling by $x\in\Q(\zeta_n)$).

(ii) Let $\Sigma$ denote the space of all $C_n$-invariant
$\Q$-valued symmetric bilinear forms on $\Q(\zeta_n)$.
$C_n$-invariance implies that bilinear forms are determined by
the $\Q$-linear map $z\mapsto(1,z)$. Moreover, $C_n$-invariance and
symmetry give that $(1,z)=(1,\bar z)$ for all $z$, so
$(1,z)=(1,\Re(z))$ and so the bilinear forms are determined by the
restriction of $z\mapsto(1,z)$ to $\Q(\zeta_n)^+$. Thus $\Sigma$
injects $\Q$-linearly into $\Hom_{\Q}(\Q(\zeta_n)^+,\Q)$. In particular,
$\dim\Sigma\leq [\Q(\zeta_n)^+:\Q]$.

Conversely, the assignment of the form
$(x,y)=\tr_{\Q(\zeta_n)/\Q}(\eta x\bar y)$ to $\eta\in\Q(\zeta_n)^+$
provides a $\Q$-linear injection $\Q(\zeta_n)^+\hookrightarrow\Sigma$.
By dimension considerations, it must be a linear isomorphism. Thus
all $C_n$-equivariant $\Q$-valued symmetric bilinear forms
must be of this form. 
The only degenerate one is when $\eta=0$, so for a non-degenerate one
we may take $\omega^{-1}=\eta$.

For this pairing, $b\in\Lambda^\myperp$ exactly when $\tr_{\Q(\zeta_n)/\Q}(\omega^{-1}b\bar a)\in\Z$ for all $a\in\a$,
which occurs precisely when $\omega^{-1}b\bar\a\subseteq\different_n^{-1}$. 
Thus $\Lambda^\myperp=\omega\different_n^{-1}\bar\a^{-1}$ as claimed.
Finally, $\Lambda^\myperp\subseteq\Lambda$ just when $\omega\different_n^{-1}\bar{\a}^{-1}\subseteq\a$, which rearranges to the desired criterion.

(iii)
As $F$ has no invariants on $\Lambda$, $\SG_{D}(\Lambda)$ is trivial by
definition. The minimal polynimial $p(t)$ is just the $n$-th cyclotomic polynomial, so the result on $p(1)$ follows from 
Lemma \ref{lem:cyc}.
By Corollary \ref{cor:BTdivP}, the order of
$\B_{\Lambda,\Lambda^\vee}$ divides $p(1)$, and by Theorems \ref{thm:bettspairing} and \ref{thm:uptosquares}(i)
its order is either a square or twice a square.
Hence $\B_{\Lambda,\Lambda^\vee}$ is trivial, unless $n=2^s$, in which
case it is either trivial or $C_2$.

In the latter case, Corollary \ref{cor:exactformulatorsionviaF} with $e=1$
tells us that $|(\Lambda/\Lambda^\vee)^F|=2/|\B_{\Lambda,\Lambda^\vee}|$.
If $[\Lambda\!:\!\Lambda^\vee]$ is odd then, then so is
$|(\Lambda/\Lambda^\vee)^F|$, and hence $\B_{\Lambda,\Lambda^\vee}\iso C_2$.
Conversely, if $[\Lambda\!:\!\Lambda^\vee]$ is even, then since
$F$ has 2-power order, it must have a non-trivial fixed point on
$(\Lambda/\Lambda^\vee)[2]$. Thus we must have
$|\B_{\Lambda,\Lambda^\vee}|=1$.

The final formula follows from Corollary \ref{cor:exactformulatorsionviaF}.
\end{proof}

\smallskip

\begin{remark}
\label{rmk:different}
The different ideal $\different_n$ is explicitly given by
$$
 \different_n=\different_{\Q(\zeta_n)/\Q}=n\cdot\!\!\!\prod_{q|n \text{ prime}}(\zeta_q-1)^{-1}.
$$
To see this, first recall that as the ring of integers of $\Q(\zeta_n)$ is $\Z[\zeta_n]$, its different is generated by $\Phi_n'(\zeta_n)$ (see e.g. \cite{samuel} remark on p 96). The cyclotomic polynomial can be expressed as
$$
\Phi_n(X)=\prod_{t|n} (X^{n/t}-1)^{\mu(t)}
$$
where $\mu$ is the M\"obius $\mu$-function, and hence, by l'H\^opital's rule
$$
 \frac{\Phi'(\zeta_n)}{n\zeta_n^{n-1}} = \prod_{1\neq t|n} (\zeta_n^{n/t}-1)^{\mu(t)} = \text{unit}\times \prod_{q|n \text{ prime}} (\zeta_q-1)^{-1},
$$
where for the second equality we have used that fact that $\mu(t)=0$ for non-squarefree $t$ and that $(\zeta_k-1)$ is a unit when $k$ is not a prime power. (The latter result follows from the identity
$N_{\Q(\zeta_k)/\Q}(1-\zeta_k)=\prod_{0<s<k, (s,k)=1}(1-\zeta_k^s)=\Phi_k(1)$, which is 1 by Lemma \ref{lem:cyc}, so that
$1-\zeta_k$ is a unit.) The claim follows.
\end{remark}

\smallskip

\begin{lemma}
\label{lem:1d}
Let $\Lambda$ be a 1-dimensional integral representation of
$C_k\!=\!\langle F\rangle$ that admits a non-degenerate $F$-invariant
symmetric $\Q$-valued pairing with $\Lambda\supseteq\Lambda^\myperp$.
Then, up to $\Z[F]$-isomorphism, $\Lambda,\Lambda^\myperp$ and $F$ are one
of the following:

\begingroup\smaller[1]
$$
\begin{array}{|l|c|c|c|c|ccl|  }
 \hline
\UDspace \text{Type} & \Lambda   & \Lambda^\myperp            & F & \SG_{D}(\Lambda) & \B_{\Lambda,\Lambda^\myperp}  & |(\Lambda/\Lambda^\myperp)^F|&    \\[2pt]
 \hline
\UDspace \voI n    & \Z          & n\Z       &1 &  1  &  1  &n&   \cr
 \hline
\UDspace \rh{\voII n}   & \rh\Z          & \rh{n\Z}       &\rh{-1}&  \rh{1}  &  C_2 &1& (n \text{ odd})     \cr
           &             &           &  &     &  1   &2& (n \text{ even})    \cr
 \hline 
\end{array}
$$
\endgroup

\medskip\noindent
The parameter $n$ is a positive integer.
All of the types arise from integral representations with a corresponding
symmetric pairing and no two are isomorphic.
The associated groups
$\SG_D(\Lambda)$ and $\B_{\Lambda,\Lambda^\myperp}$ are as described in the
table (with $D=F\!-\!1$).
\end{lemma}

\begin{proof}
As $\Lambda\iso\Z$ as an abelian group, the automorphism $F$ is either multiplication by 1 or by $-1$. The pairing can be any $\Q$-valued non-zero pairing, so, in particular, $\Lambda^\vee$ can be any sublattice of $\Lambda$. The invariants $\SG_D(\Lambda), \B_{\Lambda,\Lambda^\vee}$ and $|(\Lambda/\Lambda^\vee)^F|$ are elementary to compute direct from the definitions.
\end{proof}

\newpage

\begin{theorem}
\label{thm:2d}
Let $\Lambda$ be a 2-dimensional integral representation of $C_k\!=\!\langle F\rangle$
that admits a non-degenerate $F$-invariant symmetric $\Q$-valued pairing with
$\Lambda\supseteq\Lambda^\myperp$.
Then, up to $\Z[F]$-isomorphism, $\Lambda,\Lambda^\myperp$ and $F$ are one
of the following:

\begingroup\smaller[1]
$$
\hskip-13mm
\begin{array}{|l|c|c|c|c|ccl|  }
 \hline\UDspace 
  \text{Type} & \Lambda         & \Lambda^\myperp            & F & \SG_{D}(\Lambda) & \B_{\Lambda,\Lambda^\myperp}  & |(\Lambda/\Lambda^\myperp)^F|&   
\\[2pt] \hline  \UDspace
 \vI n,m   & \Z\oplus\Z      & n\Z\oplus m\Z        &\smallmatrix1001&  1  &  1  &nm&   
\\[2pt] \hline  \UDspace
 \rh{\vIIa n,m} & \rh{\Z\oplus\Z}      & \rh{n\Z \oplus m\Z}       &\rh{\smallmatrix100{-1}}&  \rh{1}  &  C_2 &n& (m \text{ odd})     \cr
           &                 &                      &    &     &  1 &2n& (m \text{ even})     
\\[2pt] \hline  \UDspace
 \vIIb n,m & \Z\oplus\Z+\langle \frac12,\frac12\rangle & n\Z\oplus m\Z+\langle \frac n2,\frac m2\rangle  &\smallmatrix100{-1}&  C_2  &  1  &n&       
\\[2pt] \hline  \UDspace
                &      &                 &                              &     &  C_2\times C_2 &1& (n, m \text{ odd})    \cr
 \vII n,m  & \Z\oplus\Z  & n\Z \oplus m\Z      &\smallmatrix{-1}00{-1}&  1  &  C_2 &2& (n,m \text{ odd/even})     \cr       
                &      &                 &                              &     &  1 &4& (n,m \text{ even})     
\\[2pt] \hline  \UDspace
 \vIII n   & \Z[\zeta_3] & (\zeta_3-1)n\Z[\zeta_3] &\cdot\zeta_3 & 1 & 1 &3&   
\\[2pt] \hline  \UDspace
 \rh{\vIV n}    & \rh{\Z[i]} & \rh{n\Z[i]} &\rh{\cdot i}& \rh{1} & C_2  &1& (n \text{ odd})   \cr
           &                    &       &        &   & 1    &2& (n \text{ even})  
\\[2pt] \hline  \UDspace
 \vVI n    & \Z[\zeta_3] & (\zeta_3-1)n\Z[\zeta_3] &\cdot\zeta_6& 1 & 1 &1&   
\\[2pt] \hline  
\end{array}\\[7pt]
$$
\endgroup
The parameters $n$ and $m$ are positive integers, with $n\equiv m$ mod 2 for type $\vIIbs$.
Types $\vI n,m$ and $\vI n',m'$ are isomorphic if and
only if $C_n\times C_m \iso C_{n'}\times C_{m'}$, and similarly for type $\vIIs$;
no other pairs are isomorphic. The associated groups
$\SG_D(\Lambda)$ and $\B_{\Lambda,\Lambda^\myperp}$ are as described in the
table (with $D=F\!-\!1$).
All of the above types arise from integral representations with a corresponding
symmetric pairing.
\end{theorem}

\begin{remark}
\label{rmk:scalebye}
(i)
Our notion of ``type'' refers to an isomorphism class of triples $(\Lambda,\Lambda^\vee,F)$, where $\Lambda\supseteq\Lambda^\vee$ are $\Z[F]$-lattices of the same rank, and $F$ an automorphism of finite order.
The first digits in the notation for the types refer to the orders of the eigenvalues of $F$, e.g. for $\vVI n$ there are two eigenvalues, both of which are primitive 6th roots of unity. They recover $\Lambda\otimes_\Z\Q$ as a $\Q[F]$-representation.

(ii) In each of the first four types in the table, the descriptions of $\Lambda$, $\Lambda^\vee$ and $F$ are, of course, written with respect to the same basis. In particular, for type $\vIIbs$, $F$ acts trivially on the first copy of $\Z$ and by $-1$ on the second copy.

(iii)
Note that to obtain $|(\Lambda/e\Lambda^\myperp)^F|$ one simply needs to
scale the parameters $n,m$ by $e$, as this will just change $\Lambda^\myperp$
to $e\Lambda^\myperp$.
\end{remark}

\begin{proof}
We may assume that $C_k$ acts faithfully. Write $V=\Lambda\otimes_\Z \Q$.

Since $F$ has a quadratic characteristic polynomial
with integer coefficients, its order is either 1,2,3,4 or 6. If the order
is 3,4, or 6 then $V$ is the faithful irreducible rational representation
of $C_k$, and the classification follows from Theorem \ref{thm:c_n};
this gives types $\vIIIs,\vIVs, \vVIs$.
Note that the corresponding class groups are trivial, so we may take
$\Lambda=\Z[\zeta_k]$ with $F$ acting by multiplication by $\zeta_k$.

If $k=1$, $F$ acts trivially and the result is clear. This gives type $\vIs$. 

When $k=2$ and both
the eigenvalues of $F$ are $-1$, $F$ acts by multiplication by~$-1$. In this
case $F$-invariance imposes no constraints on the pairing or
on $\Lambda^\myperp$, so the latter may be any sublattice of $\Lambda$.
Thus with respect to some basis we have $\Lambda=\Z\oplus\Z$ and
$\Lambda^\myperp=n\Z\oplus m\Z$ for some $n,m\ge 1$. As $\B$, $\SG$ and
$|(\Lambda/\Lambda')^F|$ behave multiplicatively under direct sums of
$\Z[F]$-lattices, their computation reduces to the case of 1-dimensional
lattices, which is done in Lemma \ref{lem:1d}. This gives type~$\vII n,m$.

When $k=2$ and the $F$-eigenvalues are +1 and -1, there are two possible cases, corresponding to whether $\SG_D(\Lambda)$ is trivial or $C_2$
(Corollary \ref{cor:sepgpbehaviour}). If $\SG_D(\Lambda)$ is trivial, then we can write $\Lambda=\Lambda_1\oplus\Lambda_{-1}$ as the direct sum of its eigenvalue $\pm 1$ sublattices; moreover by Theorem \ref{lem:perppreservessepgps} $\SG_D(\Lambda^\vee)$ is also trivial and hence $\Lambda^\vee=n\Lambda_1\oplus m\Lambda_{-1}$ for some $n,m,\in\Z$. The result then follows from the multiplicativity of $\B$ and $\SG$ under direct sums of lattices, as before. This gives type $\vIIa n,m$.

Finally, suppose that $k=2$, the $F$-eigenvalues are 1 and -1, and that $\SG_D(\Lambda)\iso C_2$. Write $V_1$ and $V_{-1}$ for the eigenvalue $\pm1$ subspaces of $V$. Then $\Lambda\cap V_1 = \Z u_1$, $\Lambda\cap V_{-1} = \Z u_{-1}$ for some $u_{\pm1}\in\Lambda$, and $\Lambda^\vee\cap V_1 = n\Z u_1$, $\Lambda^\vee\cap V_{-1}
=m\Z u_{-1}$ for some $n,m\in \Z$. As $\SG_D(\Lambda)\iso C_2\iso\SG_D(\Lambda^\vee)\iso C_2$ by Theorem \ref{lem:perppreservessepgps}, we must have $\Lambda = u_1\Z + u_{-1}\Z + \frac{u_1+u_{-1}}2 \Z$ and $\Lambda^\vee= nu_1\Z + mu_{-1}\Z + \frac{nu_1+mu_{-1}}2 \Z$. By Corollary \ref{cor:BTdivP}, $\B_{\Lambda,\Lambda^\vee}$ must be trivial, and by 
Corollary \ref{cor:exactformulatorsionviaF}
or by direct computation $|(\Lambda/\Lambda^\vee)^F|=n$. This gives type $\vIIb n,m$.
Note that as $\Lambda^\vee\subseteq\Lambda$ we must have $n\equiv m\mod 2$. Conversely, if $n\equiv m\mod 2$ then taking the pairing with matrix $\smallmatrix{\pm \frac 2n}00{\pm\frac2m}$ with respect to the $\Q$--basis $u_{\pm1}$ and $\Lambda = u_1\Z + u_{-1}\Z + \frac{u_1+u_{-1}}2 \Z$ gives the desired $\Lambda^\vee$, and so the only constraint on $n$ and $m$ is that $n\equiv m\mod 2$.

\end{proof}


\section{Tamagawa numbers of semistable abelian varieties}
\label{s:tam}

We now turn to the study of the behaviour of Tamagawa numbers of semistable abelian varieties in finite extensions of $p$-adic fields. Raynaud's parametrisation \cite{Ray}, \cite{Gro}\S9--10 allows us to translate this into a question about lattice quotients of the type investigated in \S\ref{s:betts}. For the convenience of the reader, we will phrase the results in a way that does not require familiarity with the preceding section, except for the willingness to use the groups $\B$ and $\SG$ as black boxes.

In \S\ref{ss:torus} we briefly review the theory of Tamagawa numbers of semistable abelian varieties and introduce the main notation. 
Our main results on Tamagawa numbers are
contained in \S\ref{ss:tam}, followed by a classification of their behaviour for abelian varieties of toric dimension 2 in \S\ref{ss:dim2}.
Finally, we turn to abelian varieties over number fields and the $p$-parity conjecture in \S\ref{sspparity}.

\subsection{Dual character group $\Lak$}
\label{ss:torus}

Here we briefly review the theory of Tamagawa numbers of semistable abelian varieties.
We only give a minimalistic description, and refer to \cite{tamroot} \S 3.5.1 for a more detailed overview, and the precise references therein for the proofs. 

\begin{notation}\label{not:lak}
Let $K/\Q_p$ be a finite extension, $A/K$ a semistable principally polarised abelian variety, and $K'/K$ an unramified extension over which $A$ acquires split semistable reduction. Let $T$ be the torus part of the Raynaud parametrisation of $A/K'$ and $X(T)$ its character group. This is a finite free $\Z$-module with an action of $\Gal(K'/K)$.
The monodromy pairing together with the principal polarisation give a non-degenerate symmetric $\Gal(K'/K)$-invariant pairing 
$(\,\cdot\,,\,\cdot\,):X(T)\times X(T)\to \Z$. 

We will write $\Lak$ for the dual lattice to $X(T)$ inside $X(T)\otimes_\Z\Q$. 
Then $(\,\cdot\,,\,\cdot\,)$ extends to a non-degenerate symmetric $\Gal(K'/K)$-invariant pairing
$$
(\,\cdot\,,\,\cdot\,):\Lak\times\Lak\to \Q,
$$ 
with $\Lak\supseteq X(T)=\Lakv$.

We will write $F\in\Gal(K^{nr}/K)$ for the Frobenius element, i.e. the element that acts as the Frobenius automorphism on the residue field. Its action on $\Lak$ factors through the quotient $\Gal(K'/K)$, so we will occasionally identify $F$ with its image in the latter group.

\end{notation}

\begin{remark}
\label{rmk:lakcv}
The group scheme of connected components of the special fibre of the N\'eron model of $A/\cO_K$ is isomorphic to $\Lak/\Lakv$, as groups with $\Gal(K^{nr}/K)$-action, and the local Tamagawa number is 
$ c_{A/K}=\Bigl|\Bigl(\frac{\Lak}{\Lakv}\Bigr)^F\Bigr|$.
Moreover, if $L/K$ is a finite extension of residue degree $f$ and ramification degree~$e$, then 
$\Lal=\Lak$ and $\Lalv=e\Lakv$, and
$$
 c_{A/L}=\biggl|\left(\frac{\Lak}{e\Lakv}\right)^{\smash{F^f}}\biggr|.
$$
\end{remark}

\begin{remark} 
Since the action of $\Gal(K^{nr}/K)$ on $\Lak$ factors through the finite quotient $\Gal(K'/K)$, the Frobenius element $F$ acts on $\Lak$ as an automorphism of finite order. Hence its action is semisimple, and in particular $\Lambda[D]=\Lambda[D^2]$ for $D=F-1$.
This is necessary to apply most of the results of \S\ref{s:betts}.
\end{remark}

\begin{definition}
We will refer to $d=\smash{\dim \Lak}$ as the {\em toric dimension} of $A/K$, 
and to $r=\dim \Lakf$ as its {\em split toric dimension}. Note that $d=r$ if and only if the reduction is split.
\end{definition}

\begin{example}\label{ex:ec1}
Suppose $E$ is an elliptic curve over $K$. If $E$ has good reduction then $\Lambda=0$, and the whole setting becomes trivial. If $A$ has multiplicative reduction of Kodaira type $\kI_n$ then $\Lak=\Z$ and $\Lakv=n\Z$, and the pairing is $(a,b)=\frac{ab}{n}$.
The Frobenius element acts either trivially on $\Lak$ or as multiplication by $-1$, corresponding to the reduction being split or non-split. The Tamagawa number can indeed be computed as 
$$
c_{A/K}=\left|\left(\frac{\Lak}{\Lakv}\right)^F\right| = |(\Z/n\Z)^F| = 
   \begin{cases}
      n & \text{split multiplicative reduction}, \\
      1 & \text{non-split multiplicative reduction, }2\!\nmid\! n  \\
      2 & \text{non-split multiplicative reduction, } 2\!\div\! n.
   \end{cases}
$$
Over a finite extension $L/K$ with ramification degree $e$ and residue degree $f$, $A/L$ 
has reduction type $\kI_{en}$ and it has split multiplicative reduction if either $A/K$ does or if $f$ is even, and non-split multiplicative reduction otherwise. This is readily seen to be compatible with the above Tamagawa number formula, 
$c_{A/L}\!=\!|(\Z/en\Z)^{F^f}|$.
\end{example}


\subsection{Behaviour of Tamagawa numbers in field extensions}
\label{ss:tam}

\begin{notation}\label{not:bak}
Let $K$ be a finite extension of $\Q_p$ and $A/K$ a semistable principally polarised abelian variety. Let $F$ be the Frobenius element of $K^{nr}/K$ and
$D=F-1\in \End(\Lak)$. Write the characteristic polynomial of $F$ acting on $\Lak$ as $\pm (t-1)^r p(t)$, where $p(1)>0$. 
We then define
$$
  \B_{A/K} = \B_{\Lambda,{\Lambda^\vee}}, \qquad\qquad \SG_{A/K}=\SG_D(\Lambda), \qquad\qquad P_{A/K}=p(1),
$$
where $\Lambda=\Lak$ with the pairing $(,)$ induced by the monodromy pairing as above.
\end{notation}

\begin{theorem}\label{thm:clptb}
Let $K/\Q_p$ be a finite extension and $A/K$ a semistable principally polarised abelian variety.
Then $\SG_{A/K}$ and $ \B_{A/K}$ are finite abelian groups and
$$
  c_{A/K} = \left|\frac{\Lakf}{\Lakvf}\right|\cdot \frac{P_{A/K}}{|\SG_{A/K}||\B_{A/K}|}.
$$
\end{theorem}

\begin{proof}
Finiteness of $\SG$ and $\B$ is established in Lemma \ref{defnlem:sepgp} and Proposition \ref{prop:bettsgpbehaviour}.
The formula is a direct consequence of Corollary \ref{cor:exactformulatorsionviaF}.
\end{proof}

\begin{theorem}\label{thm:varyfield}
Let $K/\Q_p$ be a finite extension 
and $A/K$ a semistable principally polarised abelian variety of toric dimension $d$ and split toric dimension $r$.
Let $n$ be the degree of the minimal extension over which $A$ acquires split semistable reduction, equivalently the order of the Frobenius element $F$ in its action on $\Lak$. For an integer $m$ let $a_m$ denote the multiplicity of
a fixed primitive $m$-th root of unity among the eigenvalues of $F$ acting on $\Lak$, 
and let $b_m= 0$ or $1$ according to whether $a_m=0$ or $a_m>0$.

(a) $P_{A/K} = \prod_{m=q^k} q^{a_m}$ where the product is taken over all prime powers.
In particular, $P_{A/K}$ remains unchanged in extensions with residue degree coprime to~$n$.

(b) $\SG_{A/K}$ has order dividing $P_{A/K}$, requires at most $\min (r,d\!-\!r)$ generators, and
has exponent dividing $\prod_{m=q^k} q^{b_m}$, where the product is taken over all prime powers.
$\SG_{A/K}$ remains unchanged in extensions with residue degree coprime to $n$.

(c) $\B_{A/K}$  has order dividing $P_{A/K}$, requires at most $(d\!-\!r)$ generators,
and has exponent dividing $\prod_{m=q^k} q^{b_m}$,  where the product is taken over all prime powers.
If $L/K$ is an extension with residue degree coprime to $n$ and with ramification degree~$e$ then 
$\B_{A/L}\cong \frac{\B_{A/K}}{\B_{A/K}[e]}$.

$\B_{A/K}$ admits a perfect antisymmetric pairing; in particular its order is either a square or twice a square. 
For a finite extension $L/K$, $\B_{A/L}$ has square order if and
only if at least one of the following holds: (i) $\B_{A/K}$
has square order, or (ii)~$[L:K]$ is even.

(d) $|\B_{A/K}|\cdot |\SG_{A/K}|$ divides $P_{A/K}$.
\end{theorem}

\begin{proof}
(a) Let $\Phi_m(t)$ denote the $m$-th cyclotomic polynomial. It is an elementary fact that $\Phi_m(1)=q$ or $1$ depending on whether $m=q^k$ is a prime power, or not (Lemma \ref{lem:cyc}). Since the roots of the characteristic polynomial $p(t)$ of $F$ in its action on $\Lak$ are all roots of unity, the claim follows by writing $p(t)$ as a product of cyclotomic polynomials. 

(d) Corollary \ref{cor:BTdivP} with $D=F-1$.

(b) Corollary \ref{cor:sepgpbehaviour} with $D=F-1$, observing that if $p_0(t)$ is the minimal polynomial of $F$ then $p_0(1)=\prod_{m=q^k} q^{b_m}$ by the
same argument as in (a). The result on the order follows from (d). The fact that $\SG_{A/K}$ is unchanged in extensions with residue degree coprime to $n$ follows from Remark \ref{rmk:Bcoprimef}.

(c) Proposition \ref{prop:bettsgpbehaviour} with $D=F-1$ shows that $\B_{A/K}$ is finite with at most $(d\!-\!r)$ generators and exponent dividing $p_0(1)=\prod_{m=q^k} q^{b_m}$.
The bound on the order follows from (d).
 
By Remark \ref{rmk:Bcoprimef}, $\B$ does not change in unramified extensions of degree coprime to $n$. 
The formula $\B_{A/L}\cong \frac{\B_{A/K}}{\B_{A/K}[e]}$.
for totally ramified extensions $L/K$ follows from Proposition \ref{prop:bettsgpvaryLambda'}.
The general case now follows, since an extension with residue degree coprime to $n$ is built up from an unramified one of degree coprime to $n$ followed by a totally ramified one of degree $e$.

$\B_{A/K}$ admits a perfect antisymmetric pairing by Theorem \ref{thm:bettspairing}.

For the final claim of (c), suppose that $L/K$ is a finite extension. We first reduce to the case that $L/K$ is unramified, as follows.
Let $M$ be the maximal intermediate field that is unramified over $K$, and let $e$ be the ramification degree of $L/K$.
Now $\B_{A/L}\cong\frac{\B_{A/M}}{\B_{A/M}[e]}$, so  by Theorem \ref{thm:uptosquares}(1) the order of $\B_{A/L}$ is a square if $e$ is even, and is the same up to squares as $\B_{A/M}$ if $e$ is odd.
Thus if $e$ is even then the claim is true, and otherwise it will suffice to prove the claim for $M/K$.
Hence we may (and will) assume that the extension $L/K$ is unramified.

Let $N/L$ be a further unramified extension such that $A/N$ has split semistable reduction. 
Pick a square integer $\epsilon$ that annihilates the three groups $\B_{A/K}, \B_{A/L}$ and $\B_{A/N}$.
Let $K_\epsilon$ be a totally ramified extension of $K$ of degree $\epsilon$, and let $L_\epsilon$ and $N_\epsilon$ be its composita 
with $L$ and $N$, respectively. Write $\sim$ for equality up to rational squares.
Corollary \ref{cor:Bram} below (whose proof does not use this part of the theorem) relates the Tamagawa numbers to $\B$ by the formula
$$
  c_{A/K_k} = |\B_{A/K}[\epsilon]| \cdot c_{A/K} \cdot \epsilon^{r} \sim c_{A/K} |\B_{A/K}|,
$$
and similarly for $L_\epsilon/L$ and $N_\epsilon/N$.

Theorem \ref{thm:uptosquares}(2) (together with Remark \ref{rmk:lakcv}) shows that
if $[L\!:\!K]$ is odd then $c_{A/L}\sim c_{A/K}$ and $c_{A/L_k}\sim c_{A/K_k}$; combining it with the above formula gives 
$|\B_{A/K}|\sim |\B_{A/L}|$, as required.
Finally, if $[L:K]$ is even, then Theorem \ref{thm:uptosquares}(2) shows that
$c_{A/L}\sim c_{A/N}$ and $c_{A/L_k}\sim c_{A/N_k}$, which together with the above formula gives
$|\B_{A/L}|\!\sim\!|\B_{A/N}|$; but $\B_{A/N}$ is trivial ($F$ acts trivially on $\Lambda_{A/N}$, so $P_{A/N}\!=\!1$ by (a) and 
$|\B_{A/N}|=1$ by (d)), so the order of $\B_{A/L}$ must be a square, as required.
\end{proof}

\begin{corollary}
\label{cor:Bram}
Let $K/\Q_p$ be a finite extension and $A/K$ a semistable principally polarised abelian variety of split toric dimension $r$.
Suppose that $A$ acquires split semistable reduction over an extension of degree $n$.
If $L/K$ has residue degree coprime to $n$ and ramification degree~$e$ then 
$$
  c_{A/L} = |\B_{A/K}[e]| \cdot c_{A/K} \cdot e^{r}.
$$
\end{corollary}

\begin{proof}
Using Remark \ref{rmk:lakcv}, Theorem \ref{thm:clptb} and Theorem \ref{thm:varyfield} (a), (b) and the first part of~(c),
$$
 c_{A/L} = \left|\frac{\Lalf}{\Lalvf}\right|\cdot \frac{P_{A/L}}{|\SG_{A/L}||\B_{A/L}|} 
             = \left|\frac{\Lakf}{e\Lakvf}\right|\cdot \frac{P_{A/K}|\B_{A/K}[e]|}{|\SG_{A/K}||\B_{A/K}|}
            = e^{r}\cdot c_{A/K} \cdot|\B_{A/K}[e]|  .
$$
\end{proof}

\begin{remark}\label{rmk:gcd}
To determine the Tamagawa number of $A$ over every extension $M/K$ 
it suffices to know $c_{A/L}$, $\B_{A/L}$ and the split toric dimension of $A/L$ just for the
 intermediate fields~\hbox{$K\!\subseteq\! L\! \subseteq \!K'$},
where $K'/K$ is the minimal unramified extension over which $A$ acquires split semistable reduction. 
Indeed, one can then simply apply Corollary \ref{cor:Bram} to the extension $M/M\cap K'$ to find $c_{A/M}$.
\end{remark}

\begin{remark}\label{rmk:cncv}
Theorem \ref{thm:varyfield} often lets one determine $\B,\SG$ and $P$. For instance, all three invariants are trivial if the reduction is split semistable, or if none of the eigenvalues of $F$ in its action on $\Lak$ have prime power order --- this condition forces $P_{A/K}=1$ by (a), and hence trivialises $\B$ and $\SG$ by (d). Moreover, $P$ can always be read off from the eigenvalues of $F$, and $\SG$ is also trivial if $F$ has no eigenvalue 1 (equivalently $A/K$ has split toric dimension 0), by (a) and (b), respectively.

Theorems \ref{thm:c_n}(iii) and \ref{thm:2d} describe these invariants in special cases. The second of these translates to a classification for the case of toric dimension 2, see Theorem \ref{thm:cvdim2} below. The first of these gives the following result: suppose $n>1$ and the eigenvalues of $F$ are precisely the set of $n$-th roots of unity (each occuring once), equivalently $\Lak\otimes_\Z\Q$ is the unique faithful rational representation of $C_n=\langle F\rangle$. Then $\SG_{A/K}$ is trivial; $P_{A/K}=q$ if $n=q^s$ is a prime power, and $P_{A/K}=1$ otherwise; $\B_{A/K}$ is trivial unless $n=2^s$ and $[\Lambda\!:\!\Lambda^\vee]$ is odd, in which case $\B_{A/K}\iso C_2$. In particular, by Theorem \ref{thm:clptb},
$$
 c_{A/K} =
   \begin{cases}
      q & \text{if}\quad n=q^s, q\text{ an odd prime}, \\
      2 & \text{if}\quad n=2^s \text{ and } [\Lambda\!:\!\Lambda^\vee] \text{ is even}, \\ 
      1 & \text{otherwise}.
   \end{cases}
$$
\end{remark}

\begin{example}\label{ex:ec2}
As in Example \ref{ex:ec1} consider an elliptic curve $A/K$ with multiplicative reduction of type $\kI_n$.
In this case $\SG_{A/K}$ is trivial, and $P_{A/K}$ is 1 or $2$ depending on whether the reduction is split or non-split, as one readily sees either from the definitions or from Theorem \ref{thm:varyfield}(a,b).
By Theorem \ref{thm:varyfield}(c), the group $\B_{A/K}$ is trivial if the reduction is split and is either trivial or $C_2$ in the non-split case.
In fact it is $C_2$ if and only if the reduction is non-split and $n$ is odd, as can be checked direct from the definition or using the final part of Remark \ref{rmk:cncv}.

Note that  $\left|\frac{\Lakf}{\Lakvf}\right|$ is either $n$ or $1$ depending on whether the reduction is split or not, and thus the expression
$$
  \left|\frac{\Lakf}{\Lakvf}\right|\cdot \frac{P_{A/K}}{|\SG_{A/K}||\B_{A/K}|} =
   \begin{cases}
      n & \text{for split multiplicative reduction}, \\
      1 & \text{for non-split multiplicative reduction, }2\!\nmid\! n  \\
      2 & \text{for non-split multiplicative reduction, }2|n,
   \end{cases}
$$
from Theorem \ref{thm:clptb} does indeed compute the local Tamagawa number in each case, cf. Example \ref{ex:ec1}.
This also recovers the formula in Remark \ref{rmk:cncv} for the case of non-split multiplicative reduction (but note that $n$ has a different meaning there).

Now consider a finite extension $M/K$ with residue degree $f$ and ramification degree~$e$. 
If $A/K$ has split multiplicative reduction, then 
$$
 c_{A/M} = en = c_{A/K}\cdot e,
$$ 
as predicted by Corollary \ref{cor:Bram} with $r=1$ and trivial $\B_{A/K}$. If $A/K$ has non-split multiplicative reduction, then it becomes split over the quadratic unramified extension $K'/K$. One readily checks that
$$
 c_{A/M} =
   \begin{cases}
       |\B_{A/K}[e]| \cdot c_{A/K} & \text{if }\> 2\!\nmid\! f, \\
       c_{A/K'}\cdot e &  \text{if }\> 2|f,
   \end{cases}
$$
again as predicted by Corollary \ref{cor:Bram}, with $r=0,1$ over $K$ and $K'$, respectively, and trivial $\B_{A/K'}$. This illustrates Remark \ref{rmk:gcd}, that to obtain $c_{A/M}$ over general extensions one just needs to know $c_{A/L}, \B_{A/L}$ and $r$ for the subfields \hbox{$K\!\subseteq\! L\! \subseteq \!K'$}, 
\linebreak 
where $K'$ is the minimal extension over which the reduction becomes split semistable. Thus in general, the cases ``$2|f$'' and ``$2\!\nmid\! f$'' will be replaced by the possible values of gcd$(f,n)$, where $n=[K':K]$.
\end{example}

\begin{corollary}\label{cor:iwastab}
Let $K/\Q_p$ be a finite extension and $A/K$ a semistable abelian variety.
If $K\subset L_1 \subset L_2 \subset \ldots$ is a tower of finite field extensions with $L_k/K$ of ramification degree $e_k$, then for all sufficiently large $k$
$$
  c_{A/L_k} = C \cdot e_k^{r_\infty}, 
$$
for some suitable constant $C\in\Q$, and where $r_{\infty}$ is the split toric dimension of $A/L_k$ for all sufficiently large $k$.
\end{corollary}

\begin{proof}

We begin by reducing to the case of principally polarised abelian varieties. Note first that $A/L$ and the dual abelian variety $A^*/L$ have the same Tamagawa number and the same split toric dimension over any extension $L/K$. 
Indeed, let $\Phi$ and $\Phi'$ denote the N\'eron component groups of $A/L$ and of $A^*/L$, respectively. They admit
a perfect $\Gal(L^{nr}/L)$-invariant pairing $\Phi\times\Phi'\to \Q/\Z$ (\cite{Gro} \S 11.4) and
$$
  c_{A/L}=\Phi^{\Gal(L^{nr}/L)} \qquad {\text{and}} \qquad   c_{A^*/L}={\Phi'}^{\Gal(L^{nr}/L)},
$$
so, in particular, $c_{A/L}=c_{A^*/L}$.
They have the same split toric dimension since they are isogenous (polarisation) and, in particular, have 
isogenous $l$-adic Tate modules.
It thus suffices to prove the theorem for the abelian variety $A^{4}\times A^{*4}$, which,
by Zarhin's trick, admits a principal polarisation. In other words, we may assume that $A/K$
is principally polarised.

We may replace $K$ by the maximal unramified extension of $K$ in $\bigcup L_k$ whose degree divides $n$,
the order of the Frobenius element $F$ in its action on $\Lak$.
Indeed, this extension is also contained in $L_k$ for all sufficiently large $k$, and this operation alters neither the $e_k$ nor $r_\infty$. This now ensures that the residue degrees of $L_k/K$ are all coprime to $n$, 
and also gives $r=r_\infty$. The result now follows from Corollary \ref{cor:Bram}, since the sequence $\B_{A/K}[e_k]$ is eventually constant.
\end{proof}

\begin{remark}
The analogue of Corollary \ref{cor:iwastab} can fail for non-semistable reduction. For example, the Tamagawa number of the elliptic curve 243a1 fluctuates between 1 and 3 in the layers of the $\Z_3$-cyclotomic tower of $\Q_3$, see \cite{megasha} Remark~5.4.
\end{remark}

\begin{theorem}\label{thm:cvuptosquares}
Let $K$ be a finite extension of $\Q_p$, 
and let $A/K$ be a semistable principally polarised abelian variety of toric dimension $d$ and split toric dimension~$r$.
If $L/K$ is a finite extension with residue degree $f$ and ramification degree $e$, then
$$
 c_{A/L}  \sim
    \begin{cases}
      c_{A/K} \cdot e^r
        & \text{if }\>2\nmid e,2\nmid f, \\
      c_{A/K}\cdot|\B_{A/K}|\cdot e^r
        & \text{if }\>2\div e,2\nmid f, \\
      c_{A/K^{nr}}\cdot e^d
        & \text{if }\>2\div f,
    \end{cases}
$$
where $\sim$ denotes equality up to rational squares.
\end{theorem}

\begin{proof}
This follows from Theorem \ref{thm:uptosquares2}, using Theorem \ref{thm:clptb} for the first two cases and noting that $c_{A/K^{nr}}=\left|\frac{\Lak}{\Lakv}\right|$ for the third case.
\end{proof}

\begin{remark}
Recall that the group $\B_{A/K}$ admits a perfect antisymmetric pairing, and so it has square order if and only if it admits a perfect alternating pairing.
Thus Theorems \ref{thm:varyfield}(c) and \ref{thm:cvuptosquares} may be reformulated in this language. (Note, however, that this is not equivalent to {\em every} perfect antisymmetric pairing on $\B_{A/K}$ being alternating.)
\end{remark}


\subsection{Classification for dimension 2}
\label{ss:dim2}

\begin{definition}
Let $K$ be a finite extension of $\Q_p$ and $A/K$ a semistable principally polarised abelian variety of toric dimension 2.
Recall that $\Lak\supseteq\Lakv$ are $F$-invariant 2-dimensional lattices, dual to each other with respect to a symmetric $\Q$-valued pairing. Such lattices are classified in Theorem \ref{thm:2d}. We will say that $A/K$ has a certain {\em reduction type} if $\Lak, \Lakv$ have this type in that classification. 

Thus the possible reduction types are $\vI n,m$, $\vIIa n,m$, $\vIIb n,m$, $\vII n,m$, $\vIII n$, $\vIV n$ and $\vVI n$, where
the parameters $n$ and $m$ are positive integers with $n\equiv m$ mod 2 for type $\vIIb n,m$.
Types $\vI n_1,m_1$ and $\vI n_2,m_2$ are the same if and only if
$C_{n_1}\times C_{m_1}\iso C_{n_2}\times C_{m_2}$, and similarly for types $\vIIs$. All other types are distinct.
Recall that the first two digits of the type name specify the orders of the eigenvalues of the Frobenius element acting on $\Lak$.
In particular, $A/K$ acquires split semistable reduction over an unramified extension of degree 1, 2, 2, 2, 3, 4 and 6 for the seven types, respectively.
\end{definition}

\begin{theorem}\label{thm:cvdim2}
Let $K/\Q_p$ be a finite extension and let $A/K$ be a semistable principally polarised abelian variety of toric dimension 2. Then its Tamagawa number depends on its type as listed in the following table:
$$
\begin{array}{|l|cl|lc|c|}
 \hline \vphantom{\int^{\int^a}}
 {\rm Type} & c_{A/K} && f=2 && f=3 
\\[1pt] \hline \UDspace
 \vI n,m & nm & &\text{unchanged}&&\text{unchanged} 
\\[1pt] \hline \UDspace
 \rh{\vIIa n,m}  & n & (m \text{ odd}) & \rh{\vI n,m} &&\rh{\text{unchanged}} \cr
       & 2n  & (m \text{ even}) &&&
\\[1pt] \hline \UDspace
 & & & \vI 2n,m/2 &(\ord_2\frac nm> 0)&\cr
 \vIIb n,m             & n    & & \vI n,m &(\ord_2\frac nm= 0)&\text{unchanged}\cr
            &     & & \vI n/2,2m &(\ord_2\frac nm< 0)&
\\[1pt] \hline \UDspace
 & 1 & (n,m \text{ odd})    &&&\cr
 \vII n,m   & 2 & (n, m \text{ odd/even})  &\vI n,m&&\text{unchanged}\cr
        & 4 & (n,m \text{ even})   &&&
\\[1pt] \hline \UDspace
 \vIII n & 3 & &\text{unchanged}&& \vI n,3n 
\\[1pt] \hline \UDspace
 \rh{\vIV n}                & 1 & (n \text{ odd}) & \rh{\vII n,n}&&\rh{\text{unchanged}}\cr
                       & 2 & (n \text{ even})  &&&
\\[1pt] \hline \UDspace
 \vVI n & 1 &  &\vIII n && \vII n,3n
\\[1pt] \hline
\end{array}
$$
Taking the base change of $A/K$ to a totally ramified extension of degree $e$ does not change its reduction type, but
scales the parameters $n$ and $m$ by $e$. An unramified extension of
degree coprime to 2 and 3 does not change the type or the parameters.
An unramified extension of degree 2 or 3 changes the type as shown in the table, corresponding to the columns ``$f=2$'' and ``$f=3$''.
\end{theorem}

\begin{proof}
By Remark \ref{rmk:lakcv}, the Tamagawa number is given by $ c_{A/K}=\Bigl|\left(\!\frac{\Lak}{\Lakv}\!\right)^{\!\! F}\Bigr|$. The classification of the Tamagawa numbers then follows from
Theorem \ref{thm:2d} with $\Lambda=\Lak$. By the same Remark, a totally ramified extension changes $\Lakv$ 
to $e\Lakv$, which in turn scales the parameters of the type by $e$ (see Remark \ref{rmk:scalebye}).

An unramified extension does not change $\Lak$ or $\Lakv$, but changes $F$ to~$F^f$. Thus the claim for unramified
extensions follows from Theorem \ref{thm:2d}, by a case-by-case analysis:
all cases are straightforward, except perhaps that of cubic extensions for types $\vIIIs$ and $\vVIs$, and quadratic extensions for $\vIIbs$.
To see the effect of a cubic unramified extension for the type $ \vIII n$ (respectively $\vVI n$),
note that $\Z[\zeta_3] / (\zeta_3-1)n\Z[\zeta_3]\iso C_n\times C_{3n}$ as abelian groups,
so it becomes $\vI n,3n$ (respectively, $\vII n,3n$), as claimed.
Similarly,
for a quadratic unramified extension for type $\vIIb n,m$, note that
$\frac{\Z\oplus\Z+\langle \frac12,\frac12\rangle}{n\Z\oplus m\Z+\langle \frac n2,\frac m2\rangle}
\iso 
C_{2n}\times C_{m/2}$ or 
$C_n\times C_m$ or 
$C_{n/2}\times C_{2m}$, depending on whether 
$\ord_{2}n>\ord_{2}m$,
$\ord_{2}n=\ord_{2}m$ or 
$\ord_{2}n<\ord_{2}m$, respectively --- this gives the claimed description.
(To see this isomorphism of abelian groups, write $n=n'2^a, m=m'2^b$ with odd $n', m'$. The quotient group is closely related to the elementary group
$\frac{\Z\oplus\Z}{n\Z\oplus m\Z}$ by the exact sequence 
$C_2\to \frac{\Z\oplus\Z}{n\Z\oplus m\Z} \to 
 \frac{\Z\oplus\Z+\langle \frac12,\frac12\rangle}{n\Z\oplus m\Z+ \langle \frac n2,\frac m2\rangle} \to C_2$,
and so it is isomorphic to one of $C_n\!\times\! C_m, C_{2n}\!\times\! C_{m/2}$ or $ C_{n/2}\!\times\! C_{2m}$. If $a=b$, then $(1,0), (0,1)$ and $(\frac12,\frac12)$ are all killed by $2^a n'm'=nm'=n'm$ in the quotient group. Thus there are no elements of order $2^{a+1}$, and hence the group is $C_n\times C_m$. If instead, without loss of generality, $a>b$, then $2^a n'm'(\frac12,\frac12)=(\frac{nm'}2,\frac{2^{a-b}n'm}2) \equiv (\frac n2,0)$ in the quotient group. So $(\frac{n'm'}2,\frac{n'm'}2)$ has order $2^{a+1}$ and the group must be $C_{2n}\times C_{m/2}$.)
\end{proof}



\subsection{Applications to the $p$-parity conjecture}
\label{sspparity}

We finally turn to abelian varieties over number fields.
Recall from \S\ref{ss:intro-p} that for an abelian variety $A$ over a number field $K$ we write $\X_p(A/K)$ for its dual $p^\infty$-Selmer group. This is a $\Q_p$-vector space whose dimension is, conjecturally, the rank of $A/K$. The $p$-parity conjecture (Conjecture \ref{conj:pparity}) states that the parity of this dimension can be read off from the global root number $w(A/K)$. More generally, if $F/K$ is a finite Galois extension and $\tau$ a self-dual complex representation of $\Gal(F/K)$, then it gives a formula for the parity of the multiplicity of $\rho$ in $\X_p(A/F)$:
$$
 (-1)^{\langle\X_p(A/F),\tau\rangle}=w(A/K,\tau),
$$
where $w(A/K,\tau)$ is the root number of the twist of $A/K$ by $\tau$, and $\langle\cdot,\cdot\rangle$ the usual inner product of characters of a finite group. (Strictly speaking, we first need to fix an embedding $\bar\Q_p\subset\C$, so as to be able to compare complex and $p$-adic representations. We do this once and for all.)

The aim of this section is to prove the $p$-parity conjecture for a specific supply $\TFKp$ of twisting representations $\tau$ (Theorem \ref{thm-parity}). The approach relies on the theory of Brauer relations and regulator constants of \cite{tamroot}.

\begin{notation}
Let $G$ be a finite group. We say that a formal $\Z$-linear combination of (conjugacy classes of) subgroups $\Theta=\sum_i n_i H_i$ is a  Brauer relation (or a $G$-relation), if
$$
 \bigoplus_i \C[G/H_i]^{\oplus n_i} \iso 0
$$
as a virtual representation, i.e. if the character $\sum_i n_i \chi_{\C[G/H_i]}=0$.

Fix a prime $p$. For a self-dual $\Q_p$-representation $\rho$ of $G$ we then define its regulator constant as
$$
  \RC_\Theta(\rho)= \prod_i \det\left( \frac1{|H_i|} ( , ) | \rho^{H_i} \right) \in \Q_p^\times/\Q_p^{\times2},
$$
where $(,)$ is any non-degenerate $G$-invariant pairing of $\rho$, and the determinant is computed on any $\Q_p$-basis of the invariant subspace $\rho^{H_i}$. This definition is known to be independent of the choices of the pairing and of the bases.

We now define $\T_{\Theta,p}$ to be the set of self-dual $\bar\Q_pG$-representations $\tau$ that satisfy
$$
 \langle \tau, \rho \rangle \equiv \ord_p \RC_\Theta(\rho) \mod 2
$$
for every self-dual $\Q_pG$ representation $\rho$.

For a Galois extension $F/K$ of number fields with Galois group $G$ we then define
$$
 \TFKp = \bigcup\nolimits_\Theta \T_{\Theta,p},
$$
where the sum is taken over all the Brauer relations of $G$.
\end{notation}

\begin{remark}
The set $\TFKp$ still remains rather mysterious. In the context of the parity conjecture and Brauer relations it should be thought of as the set of ``computable representations''. If, as above, $F/K$ is a Galois extension with Galois group $G$ and $A/K$ an abelian variety, then $\X_p(A/F)$ can be decomposed into irreducible $\Q_p$-representations $\X_p(A/F)\iso \bigoplus_i \rho_i^{\oplus n_i}$.
Ideally, we would like to be able to determine the multiplicities $n_i$, but this appears to be beyond reach at present.
However, the machine of Brauer relations and regulator constants gives a formula for \hbox{$\sum_{i\in I} n_i \!\mod 2$} for suitable sets $I$ of representations of $G$. These sets $I$ are determined by regulator constants: if $\Theta$ is a $G$-relation then the 
\hbox{$\Q_p$-representations}~$\rho$ for which $\ord_p\RC_\Theta(\rho)$ is odd form such a set. 
Now pick one $\bar\Q_p$-irreducible consitutent~$\tau_i$ of each of the $\rho_i$ for $i\in I$ and set $\tau=\bigoplus\tau_i$. The sum of the multiplicities $\sum n_i\!\mod 2$ that we can compute is the same as $\langle \tau,\X_p(A/F)\rangle\!\mod 2$. All the possible $\tau$ that can be obtained in this fashion are exactly the set $\TFKp$. (Technically, to get all of $\TFKp$, we also need to allow to take an odd number of constituents of $\rho_i$, and an even number of constutuents for those self-dual $\rho$ for which $\ord_p\RC(\rho)$ is even.)
\end{remark}

\begin{example}
The group $G=S_3$ has three complex irreducible representations ($\triv$, sign $\epsilon$ and 2-dimensional $\rho$) and four subgroups up to conjugacy ($S_3, C_3, C_2, \{{\rm id}\}$). The corresponding permutation representations decompose as 
$$
\C[G/S_3]\iso\triv \qquad \C[G/C_3]\iso\triv\oplus\epsilon \qquad \C[G/C_2]\iso\triv\oplus\rho \qquad 
\C[G/\{{\rm id}\}]\iso\triv\oplus\epsilon\oplus\rho^{\oplus 2}.
$$
Thus one easily sees that, up to multiples, $S_3$ has exactly one Brauer relation: 
$$
  \Theta=2S_3-C_3-2C_2+\{{\rm id}\}.
$$

Now fix any prime $p$. The irreducible $\Q_p$-representations of $S_3$ are still $\triv, \epsilon$ and~$\rho$, since they can all be realised over $\Q$. To compute their regulator constants we need to pick an $S_3$-invariant pairing on the representations. For example, for $\triv$ we can take the pairing $(1,1)=1$, or indeed any other one. We then compute
$$
 \RC_\Theta(\triv) = \left(\frac{1}{6}\cdot 1 \right)^2  \left(\frac{1}{3}\cdot 1 \right)^{-1} 
 \left(\frac{1}{2}\cdot 1 \right)^{-2}  \left(\frac{1}{1}\cdot 1 \right) = 3 \cdot \square.
$$
Note that choosing a different pairing would not have affected the result (the powers neatly cancel!) and that picking different bases for the invariant spaces $\triv^{H}$ for the various subgroups $H$ would only have changed the result by a square in $\Q_p$.
One similarly computes the regulator constant for the other two irreducibles, which in this example give the same result: $\RC_\Theta(\epsilon)=\RC_\Theta(\rho)=3$.

Now let $p=3$. The parity of the power of $3$ in the regulator constant of a representation $X$ can simply be computed by counting the number of $\triv,\epsilon$ and~$\rho$ in its decomposition into irreducibles (regulator constants are multiplicative). 
Equivalently, $\ord_3 \RC_\Theta(X)\equiv\langle\triv\oplus\epsilon\oplus\rho,X\rangle \!\mod \! 2$.
This precisely means that $\tau=\triv\oplus\epsilon\oplus\rho\in\T_{\Theta,3}$.

The purpose of this machine is that we can prove the $3$-parity conjecture for twists by representations like $\tau$: 
Theorem \ref{thm-parity} shows that if $F/K$ is an extension of number fields with Galois group $S_3$, then the $3$-parity conjecture holds for a large class of abelian varieties over $K$ twisted by the representation $\tau$ of $\Gal(F/K)$.

\end{example}

\begin{notation}
For an abelian variety over a local field $A/K$ together with a non-zero regular exterior form $\omega$, we write
$$
 C(A/K,\omega)= 
    \begin{cases}
      c_{A/K} \cdot \left|\frac{\omega}{\neron}\right|_K  & \text{ for } K \text{ non-archimedean,}\\
     \int_{A(K)}|\omega| & \text{ for } K=\R, \\
     2^{\dim A}\int_{A(K)}|\omega\wedge\bar\omega| & \text{ for } K=\C,
    \end{cases}
$$
where $|\cdot|_K$ is the normalised absolute value of $K$ and $\neron$  is the N\'eron exterior form.
If $F/K$ is a Galois extension and $\Theta=\sum_i n_iH_i$ a Brauer relation in its Galois group, we use the shorthand notation
$$
  C_v(\Theta)= \prod_i C(A/F^{H_i},\omega)^{n_i}.
$$
This quantity is independent of the choice of $\omega$ (see \cite{tamroot} Notation 3.1).
The subscript ``$v$'' is there only to indicate that the setting is local.

If $K$ is instead a number field, we write $C_{A/K}=\prod_v C(A/K_v,\omega)$, the product taken over all the places of $K$. This definition is independent of the choice of $\omega$ by the product formula. We similarly write
$$
  C(\Theta)= \prod_i C(A/F^{H_i})^{n_i},
$$
whenever $F/K$ is a Galois extension and $\Theta=\sum_i n_iH_i$ a Brauer relation in its Galois group.
\end{notation}

\begin{definition}
Let $K$ be a local field, $A/K$ an abelian variety and $p$ a fixed prime
number. For a Galois extension $F/K$, let us say that
{\em local compatibility holds for $A$ in $F/K$} if
for every Brauer relation $\Theta$ of $\Gal(F/K)$ and every
$\tau\in\T_{\Theta,p}$,
$$
  (-1)^{\ord_p C_v(\Theta)} = w(A/K,\tau).
$$
We will also say that {\em local compatibility holds} for $A/K$ if it
holds for $A$ in all Galois extensions $F/K$.
\end{definition}

\begin{theorem}
\label{comp-parity}
Let $F/K$ be a Galois extension of number fields, $p$ a prime number,
and $A/K$ and abelian variety.
Suppose that local compatibility holds for $A$ in $F_w/K_v$ at
every place $v$ of $K$ and every place $w|v$ of $F$. Then
$$
  (-1)^{\ord_p C(\Theta)} = w(A/K,\tau)
$$
for every $\Gal(F/K)$-relation $\Theta$ and every $\tau\in\T_{\Theta,p}$.

Suppose further that $A$ is principally polarised, and that, if $p=2$, the
polarisation comes from a $K$-rational divisor. Then
the $p$-parity conjecture holds for all twists of $A/K$ by $\tau\in\TFKp$.
\end{theorem}

\begin{proof}
This is proved in the beginning of \S3 of \cite{tamroot}, although unfortunately
not stated there in this form. We will not repeat the proof here, as it
requires quite a lot of notation and several general results from \cite{tamroot}.
We will just explain which parts need to be taken and how they need to be
modified.

The first part of the theorem follows from the proof of \cite{tamroot} Cor.\ 3.4.:
the single use in that proof of \cite{tamroot} Thm.\ 3.2.\ (in the form of Cor.\ 3.3.)
is replaced by our hypothesis that local compatibility holds for $A$
in $F_w/K_v$ at all $v$. The rest of the proof of \cite{tamroot} Cor.\ 3.4.\ holds
verbatim.

The second part of the theorem follows from the proof of \cite{tamroot} Thm.\ 1.6.,
(after the proof of Cor.\ 3.4). It is a simple application of \cite{tamroot} Thm.\ 1.14.
\end{proof}

\begin{remark}
The assumption on the $K$-rational divisor can be relaxed to 
$\sha^\circ(A/L)[p^\infty]$ having square order for every $K\subseteq L\subseteq F$, or to the even weaker
requirement that $\prod_i |\sha^\circ(A/F^{H_i})[p^\infty]|^{n_i}$ is a square for every $G$-relation \hbox{$\Theta=\sum n_i H_i$}.
\end{remark}

\begin{lemma}
\label{comp-cyc}
Local compatibility holds for all abelian varities in cyclic extensions
of local fields.
\end{lemma}
\begin{proof}
Cyclic groups have no Brauer relations, so there is nothing to prove.
\end{proof}

\begin{theorem}
\label{comp-ss}
Local compatibility holds for all semistable principally polarised
abelian varieties over local fields of characteristic zero.
\end{theorem}

\begin{proof}
For archimedean local fields this follows from the previous lemma.

Suppose $F/K$ is a finite Galois extension of $l$-adic fields.
Let $p$ be a prime number, $A/K$ a semistable principally polarised
abelian variety, $\Theta\!=\!\sum n_i H_i$ a \hbox{$\Gal(F/K)$-relation} and~$\tau\in \T_{\Theta,p}$.
Write $d$ and $r$ for the toric dimension and the split toric dimension of $A/K$, respectively, and $\Lambda=\Lak\otimes_\Z\Q_p$.

The root number of the twist is given by
$$ 
  w(A/K,\tau)=
  w(\tau)^{2\dim A} (-1)^{\langle \tau, \Lambda \rangle}=
  (-1)^{\langle \tau, \Lambda \rangle} ,
$$
where the first equality is a standard formula for the root number of a twist of a semistable abelian variety by a self-dual representation, and the second follows from the determinant formula $w(\sigma)w(\sigma^*)=\det\sigma(-1)$ and the fact
that all elements of $\T_{\Theta,p}$ have trivial determinant, see \cite{tamroot} Prop. 3.23., Lemma A.1 and Thm.\ 2.56.

Let $\omega$ be a minimal exterior form on $A/K$. As $A/K$ is semistable, $\omega$ remains minimal over every extension $L/K$, so that 
$$
C(A/L,\omega)=c_{A/L} =
   \begin{cases}
      c_{A/K} \cdot e^r \cdot \square
        & \text{if }\>2\nmid e,2\nmid f, \\
      c_{A/K}\cdot|\B_{A/K}|\cdot e^r\cdot \square
        & \text{if }\>2\div e,2\nmid f, \\
      c_{A/K^{nr}}\cdot e^d\cdot \square
        & \text{if }\>2\div f,
  \end{cases},
$$
by Theorem \ref{thm:cvuptosquares}, where $e$ and $f$ denote the ramification and the residue degree of $L/K$, respectively.

We now proceed to show that $(-1)^{\ord_p C_v(\Theta)}=w(A/K,\tau)$ by using these explicit formulas for both terms.
Rather than doing the computation for the behaviour of these functions in Brauer relations from scratch,
we will take a shortcut by making use of the fact that the theorem has already been established for semistable elliptic curves in \cite{tamroot} Prop. 3.9. If $E_1, E_2$ and $E_3$ are elliptic curves with, respectively, split multiplicative reduction of type $\kI_1$
and non-split multiplicative reduction of types $\kI_1$ and $\kI_2$, then their Tamagawa numbers over $L$ are
$$
 c_1(L)=c_{E_1/L}=e,\>\>\>\quad
 c_2(L)=c_{E_2/L}= \leftchoicethree
           {1}{\text{if}\> 2\nmid e, \>2\nmid f}
           {2}{\text{if}\> 2|e, \>2\nmid f}
           {e}{\text{if}\> 2|f},\>\>\>\quad
 c_3(L)=c_{E_3/L}= \leftchoice
           {2}{ \text{if}\> 2\nmid f}
           {2e}{\text{if}\> 2|f}.
$$
Thus we already know that
$$
 \ord_p c_1(\Theta)\equiv \langle \tau, \triv \rangle,\>\quad
 \ord_p c_2(\Theta)\equiv \langle \tau, \chi \rangle,\>\quad
 \ord_p c_3(\Theta)\equiv \langle \tau, \chi \rangle,
$$
where $c_j(\Theta)=\prod_i (c_j(F^{H_i})^{n_i}$, 
$\chi$ denotes the unramified character of order 2 of $K$, and $\equiv$ is equality mod~2.

Since the Galois group acts on $\Lambda$ through a finite cyclic quotient, and this representation is unramified and self-dual,
$$
 \Lambda\iso \triv^{\oplus r} \oplus \chi^{\oplus s} \oplus (\rho\oplus\rho^*)
$$
for some representation $\rho$ and some integer $s\equiv d-r\mod 2$.

Note that we can factor
$$
 C(A/L,\omega) = \gamma \left(\frac{2c_2}{c_3}\right)^\eta c_1^r \left(\frac{c_3}{2}\right)^{d-r} \cdot \square,
$$
where 
$\gamma= \gamma(L) =  \leftchoice
       {c_{A/K}}{ \text{if}\> 2\nmid f}
       {c_{A/{K^{nr}}}}{\text{if}\> 2|f}$
and $\eta$ is 0 or 1 depending on whether $|\B_{A/K}|$ is a square or twice a square (Theorem \ref{thm:varyfield}(c)).
Set $\gamma'(L)=2^{\eta-d+r}\gamma(L) $ and write $\gamma'(\Theta)=\prod_i\gamma'(F^{H_i})^{n_i}$. Then
$$
  C_v(\Theta) = \gamma'(\Theta) \left(\frac{c_2(\Theta)}{c_3(\Theta)}\right)^\eta c_1(\Theta)^r c_3(\Theta)^{d-r} \cdot \square.
$$
As a function of the field $L$, $\gamma'$ depends only on the residue degree $f$, and hence $\gamma'(\Theta)=1$; this is a
general fact about evaluating such functions on Brauer relations, see \cite{tamroot} Thm.\ 2.36(f). Hence
$$
 \ord_p C_v(\Theta) \equiv r\langle \tau, \triv \rangle + (d-r) \langle \tau, \chi \rangle
   \equiv \langle\tau, \triv^{\oplus r}\oplus \chi^{\oplus s}\rangle
   \equiv \langle\tau, \Lambda\rangle \mod 2,
$$
because $s\equiv d-r\mod 2$ and $\tau$ is self-dual. By the formula for the root number above,
$$
 (-1)^{\ord_p C_v(\Theta)}= w(A/K.\tau),
$$
as required.
\end{proof}

\begin{theorem}
\label{thm-parity}
Let $F/K$ be a Galois extension of number fields and let $p$ be a prime
number. Let $A/K$ be a principally polarised abelian variety
all of whose primes of unstable reduction have cyclic
decomposition groups in $F/K$; if $p=2$ assume also that the
polarisation is induced by a $K$-rational divisor.
Then the $p$-parity conjecture holds for all twists of $A/K$
by $\tau\in \T^{\scriptscriptstyle{F/K}}_p$. 
\end{theorem}

\begin{proof}
This follows from Lemma \ref{comp-cyc}, Theorem \ref{comp-ss} and
Theorem \ref{comp-parity}.
\end{proof}


\def\pap{{A^+}}
\def\pam{{A^-}}
\def\pbp{{B^+}}
\def\pbm{{B^-}}
\def\pgp{{C^+}}
\def\pgm{{C^-}}

\def\vna{{a}}
\def\vnb{{b}}
\def\vng{{c}}
\def\vnd{{d}}
\def\vnn{{n}}
\def\cC{{\mathcal C}}
\newmathop{Jac}
\def\Ljk{{\Lambda_{\scriptstyle J/K}}}


\bigskip

\section{Appendix: Integral module structure of $\Lak$ for Jacobians of semistable hyperelliptic curves of genus 2}
\label{s:appendix}
\setcounter{subsection}{1}
\setcounter{thmcount}{0}

\medskip

\begin{center}
by Vladimir Dokchitser and Adam Morgan
\end{center}

\bigskip

In this appendix we explain how the ``lattice type'', in the sense of Theorems~\ref{thm:introclassification} and \ref{thm:2d}, of the dual character group  $\Ljk$ of the toric part of the reduction can be determined when $J$ is the Jacobian of a semistable hyperelliptic curve of genus~2. 
Theorem~\ref{thm:introclassification} (or Theorem~\ref{thm:cvdim2}) then lets one read off the Tamagawa number of~$J$ over all finite extensions of the base field.

Let $p$ be an odd prime and $K/\Q_p$ a finite extension. Let $\cC/K$ be a hyperelliptic curve given by an equation
$$
\cC:\qquad y^2=f(x),
$$
with $f(x)\in\cO_K[x]$ with no repeated roots in $\bar K$. Write $J=\Jac(\cC)$ for the Jacobian of $\cC$.
The article \cite{DDMM} provides an explicit criterion that determines whether $\cC$ is semistable, and, under this assumption, a description of its minimal regular model and its special fibre (together with the Frobenius action) and of the lattice $\Lak$. Here we apply this machinery  in the case when $\cC$ has genus 2 to obtain the lattice type of $\Ljk$ in terms of elementary data attached to $f(x)$. We omit the proofs and computations, which will be included in \cite{DDMM}.

The description is similar to the familiar one for elliptic curves with multiplicative reduction, corresponding to the case when $f(x)$ is a cubic whose reduction has a
double root in $\overline\F_p$. There one checks the tangents at the singular point to decide whether the reduction is split or non-split multiplicative, and the valuation of the discriminant (equivalently, twice the valuation of the difference of the two roots that have the same reduction) to determine the parameter $n$ of the Kodaira type~ $\kI_n$. 

For the purposes of this appendix we assume that $f(x)$ has degree 5 or 6, so that~$\cC$ has genus 2, and impose the following simplifying assumption:
\begin{itemize}
 \item $f(x)$ has a unit leading coefficient and its reduction $\bar f(x)$ has no triple \\ roots in $\bar\F_p$. 
(In particular, $\cC$ and $J$ have semistable reduction.)
\end{itemize}

For a double root $\alpha\in\bar\F_p$ of $\bar f(x)$ define the two ``tangents'' $t_\alpha^{\pm}$ by
$$
  t_\alpha^{\pm} = \pm \sqrt{g(\alpha)}, \qquad {\rm where} \qquad \bar f(x) = (x-\alpha)^2g(x),
$$
and define the pairs
$$
  A^{\pm}= (\alpha,t_\alpha^{\pm}).
$$
Also set 
$$
 \vna=2\cdot v_K(\alpha_1\!-\!\alpha_2)\in \Z,
$$ where $\alpha_1,\alpha_2\in\bar K$ are the two roots of $f(x)$ that reduce to $\alpha$, and the valuation $v_K:\bar K^\times\to \Q$ is normalised so as to send the uniformiser of $\cO_K$ to 1. 
If there are further double roots of $\bar f(x)$ present, say $\beta$ and $\gamma$, we similarly define the quantities $t_\beta^{\pm}, t_{\gamma}^\pm$, $B^\pm, C^\pm$ and $\vnb, \vng$, corresponding to these roots.

The type of the lattice $\Ljk$ can be read off from the valuations $a, b, c$ and the action of the Frobenius automorphism on $A^\pm, B^\pm, C^\pm$ as shown in the following table. For convenience of the reader, the table also lists the order of Frobenius in this action, the Tamagawa number $c_{J/K}$ and the toric dimension of $J$, equivalently the rank of $\Ljk$.
As in Lemma \ref{lem:1d},
types $\voI \vna$ and $\voII \vna$ denote the lattice~$\Lambda\!=\!\Z$ with dual lattice $\Lambda^\vee=a\Z$ and $F$ acting as multiplication by $+1$ and $-1$, respectively; the types of 2-dimensional lattices are described in Theorems \ref{thm:introclassification} and \ref{thm:2d}.
We use the shorthand notation 
$$
  \overline{x}=1 \>\text{ if } x \text{ is odd,} \qquad   \overline{x}=2 \>\text{ if } x \text{ is even},
$$
and
$$
 \vnd=gcd(\vna,\vnb,\vng), \qquad\qquad \vnn=\vna\vnb+\vnb\vng+\vna\vng,
$$ 
in the case when $\bar f$ has three double roots.

\begingroup\smaller[1]
$$
\hskip -1.5cm
\begin{array}{|c|c|l|c|c|c|}
 \hline \vphantom{\int^{\int^a}}
 {\rm Double} & {\text{Order of}}& {\text{Frobenius orbits on pairs}} & {\rm Toric} & {\text{Lattice type}} & {\text{Tamagawa}}\cr
 {\text{roots of }} \bar{f} & {\rm Frobenius}&  & {\rm dimension} & {\text{of }} \Ljk& {\text{number}}\cr 
\hline\UDspace{\vphantom {\int^{\int^\int}}}
{\text{none}}& - & - &  0 & - & 1\cr 
\hline\UDspace{\vphantom {\int^{\int^\int}}}
 
\alpha & 1 & \{\pap\}\quad \{\pam\}&1& \voI \vna & a \cr
\alpha & 2 & \{\pap,\pam\}       &1& \voII \vna& \overline{a} \cr
\hline\UDspace{\vphantom {\int^{\int^\int}}}

\alpha,\beta & 1 &\{\pap\}\quad \{\pam\}\quad \{\pbp\}\quad \{\pbm\}&2&\vI \vna,\vnb & a\cdot b\cr
\alpha,\beta & 2 &\{\pap,\pam\}\quad \{\pbp\}\quad \{\pbm\}&2& \vIIa \vna,\vnb & a\cdot \overline{b} \cr
\alpha,\beta & 2 &\{\pap,\pam\}\quad \{\pbp,\pbm\}&2& \vII \vna,\vnb & \overline{a}\cdot \overline{b}\cr
\alpha,\beta & 2 &\{\pap,\pbp\}\quad \{\pam,\pbm\}&2& \vIIb \vna, \vna & a\cr
\alpha,\beta & 4 &\{\pap,\pbp,\pam,\pbm\}&2& \vIV \vna & \overline{a} \cr
\hline\UDspace{\vphantom {\int^{\int^\int}}}

\alpha, \beta, \gamma & 1 & \{\pap\}\quad\{\pam\}\quad \{\pbp\}\quad \{\pbm\}\quad \{\pgp\}\quad \{\pgm\}&2 & \vI \vnd, {\vnn}/{\vnd}& n\cr
\alpha, \beta, \gamma & 2 & \{\pap,\pam\}\quad \{\pbp,\pbm\}\quad \{\pgp,\pgm\}&2 & \vII \vnd, {\vnn}/{\vnd} & \overline{d}\cdot \overline{n/d}\cr
\alpha, \beta, \gamma & 2 & \{\pap,\pbp\}\quad \{\pam,\pbm\}\quad \{\pgp\}\quad \{\pgm\}&2 & \vIIb \vna\!+\!2\vng, \vna & a\!+\!2c\cr
\alpha, \beta, \gamma & 2 & \{\pap,\pbp\}\quad \{\pam,\pbm\}\quad \{\pgp, \pgm\}&2 &\vIIb \vna, \vna\!+\!2\vng & a\cr
\alpha, \beta, \gamma & 3 & \{\pap,\pbp,\pgp\}\quad \{\pam,\pbm,\pgm\}&2 & \vIII \vna & 3\cr
\alpha, \beta, \gamma & 6 & \{\pap,\pbp,\pgp,\pam,\pbm,\pgm\}&2 & \vVI \vna & 1\cr

 \hline
\end{array}
$$
\endgroup

\bigskip

Let us remark that the table is complete, in the sense that these are all the possible Frobenius actions on the pairs under our hypotheses on $f(x)$.
It is also not too difficult to check that by varying $f(x)$ it is all possible to obtain all the lattice types of Theorem \ref{thm:introclassification} with all possible parameter values.

\begin{example}
Let $\cC/\Q_3$ be given by the equation
$$ 
  y^2= f(x) 
$$
with
$$
 f(x)= x^5+x^4+20x^3+20x^2+64x+64.
$$
Over the residue field, this polynomial factorises as $\bar{f}(x)=(x+i)^2(x-i)^2(x+1)$, where $i\in\F_9$ denotes a square root of $-1$.
It visibly has the double roots $\alpha =i$ and $\beta =-i$, along with the single root $-1$. Moreover
$$
 t^{\pm}_\alpha = \pm\sqrt{(\alpha+i)^2(\alpha+1)}= \pm iw, \quad\qquad
 t^{\pm}_\beta  = \pm\sqrt{(\beta-i)^2(\beta+1)}= \pm iw^3,
$$
where $w$ denotes a square root of $1\!+\!i$ in $\F_{81}$ (a 16th root of unity). Thus the pairs we need to determine the reduction type are $A^{\pm}=(i,\pm iw)$ and $B^{\pm}=(-i,\pm iw^3)$. Frobenius acts on these as the 4-cycle $(A^+ B^- A^- B^+)$, so the the table gives the associated lattice type as $\vIV a$ for some $a\in\Z$.

Over $\overline\Q_3$, the polynomial factors as $f(x)= ((x-i)^2 +9)( (x+i)^2+9)(x+1)$, where $i$ now denotes a square root of $-1$ in $\overline\Q_3$. Writing $\alpha_1,\alpha_2$ for the roots that reduce to $\alpha$, we find that
$$
 a = 2 \cdot v_{\Q_3}(\alpha_1\!-\!\alpha_2)=v_{\Q_3}(\Disc((x\!-\!i)^2\!+\!9)) = v_{\Q_3}(-36)=2,
$$
so the lattice type is $\vIV 2$.

In particular, if $J/\Q_3$ is the Jacobian of $\cC$ and $K/\Q_3$ a finite extension of residue degree $f$ and ramification degree $e$, it follows from Theorem \ref{thm:introclassification} that the Tamagawa number of $J/K$ is given by
$$
c_{J/K}=\begin{cases}
2 & {\text{if }}\> 2\!\nmid\! f,\\
4 & {\text{if }}\> 2\!\mid \mid\! f,\\
4e^2 & {\text{if }}\> 4 \!\mid\! f.
\end{cases}
$$

\end{example}

\begin{example}
Let $\cC/\Q_3$ be the curve given by
$$
 y^2=f(x)=x^6+17x^4+76x^2+36 .   
$$
The polynomial factors as $f(x)=( (x-i)^2+3)( (x+i)^2+3)(x^2+9)$ over $\overline\Q_3$, so its reduction $\bar{f}$ has 3 double roots, $\alpha =i, \beta =-i, \gamma= 0$. The corresponding tangents are $t_{\alpha}^{\pm}= \pm 2, t_{\beta}^{\pm} = \pm 2$ and $t_{\gamma}^{\pm} = \pm 1$, giving the pairs 
$$
A^{\pm} = (i,\pm 2), \qquad B^{\pm} = (-i,\pm 2), \qquad C^{\pm}= (0, \pm 1).
$$ 
The Frobenius orbits on these pairs are 
$\{A^+,B^+\}, \{A^-,B^-\}, \{C^+\}, \{C^-\}$ and, as in the previous Example, one easily checks that the (scaled) distances in $\overline\Q_3$ between the pairs of roots that reduce to $\alpha, \beta$ and $\gamma$ are $a\!=\!b\!=\!1$ and $c\!=\!2$. The table now shows the lattice type to be $\vIIb 5,1$. 

If $J/\Q_3$ is the Jacobian of $\cC$ and $K/\Q_3$ a finite extension of residue degree $f$ and ramification degree $e$, it follows from Theorem \ref{thm:introclassification} that the Tamagawa number of~$J/K$ is given by
$$
c_{J/K}=\begin{cases}
  5e      &{\text{if }}\> 2\!\nmid\! f,\\
  5e^2 &{\text{if }}\> 2\!\mid\! f.
\end{cases}
$$
\end{example}

\end{document}